\documentclass[amsfonts,11pt]{amsart}
\RequirePackage{amsmath,amssymb,amsthm, amscd, comment,mathtools}
\usepackage[margin=1.2in]{geometry}
\usepackage{amsmath}
\usepackage{amsthm}
\usepackage{amssymb}
\usepackage{amscd}
\usepackage{mathrsfs}
\usepackage{graphicx}
\usepackage{rotating, graphicx}
\usepackage[all,cmtip]{xy}
\usepackage{enumerate}
\usepackage{tabularx}
\usepackage{tikz-cd}
\usepackage[thinlines]{easytable}
\usepackage{multirow}
\usepackage{mathtools}
\usepackage{mathrsfs}
\usepackage{changepage} % for \adjustwidth
\usepackage{paralist}%\usepackage{enumerate}

\usepackage{kantlipsum}
\usepackage{mathabx}
\usepackage{enumitem}
\usepackage{times}
\usepackage{subfig}
\usepackage{color}
\usepackage[pagebackref,breaklinks,colorlinks,linkcolor=red,anchorcolor=red,citecolor=blue]{hyperref}

\calclayout

\newcommand{\SH}{\mathbf{SH}}
\newcommand{\DK}{\mathbf{DK}}
\newcommand{\KM}{\mathbf{KM}}
\newcommand{\KC}{\mathbf{KC}}
\newcommand{\KGL}{\mathbf{KGL}}
\newcommand{\KBM}{K^{\mathrm{BM}}}

\newcommand{\Tor}{\operatorname{Tor}}

\newcommand{\one}{\mathbf{1}}

\newcommand{\A}{\mathbb{A}}

\newtheorem{proposition}[subsection]{Proposition}
\newtheorem{corollary}[subsection]{Corollary}
\newtheorem{theorem}[subsection]{Theorem}
\newtheorem{lemma}[subsection]{Lemma}

\theoremstyle{definition}
\newtheorem{definition}[subsection]{Definition}
\theoremstyle{remark}

\newtheorem{remark}[subsection]{Remark}

\numberwithin{equation}{subsection}

\DeclareMathAlphabet{\mathbbold}{U}{bbold}{m}{n}

\title[]{$K_0$-motives and the Chow weight-heart}

\author{Fangzhou Jin}
\address{School of Mathematical Sciences\\
Key Laboratory of Intelligent Computing and Applications (Ministry of Education)\\
Tongji University\\
Siping Road 1239\\
200092 Shanghai\\
China}
\email{\href{mailto:fangzhoujin@tongji.edu.cn}{fangzhoujin@tongji.edu.cn}}
\urladdr{\url{https://fangzhoujin.github.io/}}

\author{Kunpeng Li}
\address{School of Mathematical Sciences\\
Key Laboratory of Intelligent Computing and Applications (Ministry of Education)\\
Tongji University\\
Siping Road 1239\\
200092 Shanghai\\
China}
\email{\href{mailto:2533770@tongji.edu.cn}{2533770@tongji.edu.cn}}

\date{\number\day-\number\month-\number\year}

\begin{document}

\maketitle

\begin{abstract}

Bondarko and Luzgarev constructed the Chow weight structure on the homotopy category of $KGL$-modules; we show that the heart of this weight structure is equivalent to the category of $K_0$-motives introduced by Gillet and Soul\'e, which is the $K$-theoretic counterpart of the category of relative Chow motives defined by Corti and Hanamura. The proof is achieved via a detailed study on the Borel-Moore theory in $KGL$-modules, as well as explicit formulas for products in algebraic $G$-theory in terms of $\operatorname{Tor}$-operations.

\end{abstract}

\setcounter{tocdepth}{1}
\tableofcontents

\noindent
\section{Introduction}

The notion of \emph{weight structure} is introduced by Bondarko as a counterpart of t-structures (see \cite{Bon10} and \cite{Bon14}), which allows to define analogues of Deligne's weight filtration and weight spectral sequence in various motivic categories (\cite{Del}). The weight-$0$ part, also known as the \emph{heart} of the weight structure, carries very interesting information: on the one hand, one should expect to define a bounded filtrations on motivic sheaves, with graded pieces living in the heart; on the other hand, even if the heart need not be an abelian category, one could recover a weight structure from its heart (\cite[Prop. 5.2.2]{Bon10}).

One of the main examples is the so-called \emph{Chow weight structure}. Let $S$ be a separated scheme of finite type over a perfect field $k$ of exponential characteristic $p$. 
H\'ebert (\cite{Heb}) and Bondarko (\cite{Bon14}) independently proved the existence of a Chow weight structure on the category of constructible cdh-motives $\operatorname{DM}_{cdh,c}(S,\mathbb{Z}[1/p])$ (\cite{CD2}). In \cite{Jin16}, the first-named author proved the following result:
\begin{theorem}[\textrm{\cite[Theorem 3.17]{Jin16}}]
%Let $S$ be a quasi-projective scheme over a perfect field $k$ of exponential characteristic $p$. 

There is an equivalence
\begin{align}
\label{eq:Jin317}
\mathbf{CHM}(S)[p^{-1}]\simeq \operatorname{Ho}(\mathbf{DM}_{cdh,c}^{w=0}(S)[p^{-1}]
\end{align}
between the $\mathbb{Z}[1/p]$-linear categories of relative Chow motives over $S$ defined by Corti-Hanamura (\cite[Def. 2.17]{CH}) and the homotopy category of the heart of the Chow weight structure on $\operatorname{DM}_{cdh,c}(S)[p^{-1}]$.
\footnote{In \cite{Jin16}, the scheme $S$ is assumed quasi-projective over $k$. This hypothesis is indeed unnecessary, as we will see below.}
\end{theorem}
The goal of this paper is to investigate a $K$-theoretic analogue of the equivalence~\eqref{eq:Jin317}. 
On the side of Chow motives, Gillet-Soul\'e defined an additive category $\KM(S)$ of \textbf{$K_0$-motives} using the intersection theory in algebraic $G$-theory (\cite[5.4]{GiS09}). We will recall this construction in Definition~\ref{def:K0M}.
On the other hand, denote by $\KGL_S$ the motivic algebraic $K$-theory spectrum (\cite{Rio10}), which has a structure of $\mathbb{E}_\infty$-ring spectrum (\cite{NSO}).
Define 
\begin{align}
\DK(S)=\operatorname{Mod}_{\KGL_S}
\end{align}
as the $\infty$-category of modules over the $\mathbb{E}_\infty$-ring spectrum $\KGL_S$. Let $\DK_c(S)\subset\DK(S)$ be the full subcategory of compact objects (which agrees with the constructible objects). In \cite{BoL16}, Bondarko-Luzgarev proved the existence of a Chow weight structure on $\DK_c(S)$ (and more generally on the whole $\DK(S)$). Our main theorem is the following:
 % as the idempotent completion of the category $KC(S)[\mathcal{S}^{-1}]$ .
\begin{theorem}[See Theorem~\ref{thm:main}]
%Let $S$ be a separated scheme of finite type over a perfect field $k$ of exponential characteristic $p$. Then t
Let $S$ be a scheme such that there is a dominant morphism of finite type $S\to B$ with $B$ an integral, separated scheme of dimension at most $3$. Let $\mathcal{P}$ be the set of all primes that are not invertible on $S$.
There is an equivalence
\begin{align}
\KM(S)[\mathcal{P}^{-1}]\simeq \operatorname{Ho}(\DK_c^{w=0}(S)[\mathcal{P}^{-1}])
\end{align}
between the category of $K_0$-motives and the homotopy category of the heart of the Chow weight structure on $\DK_c(S)$.
\end{theorem}
%Note that if $k$ is a perfect field that satisfies embedded resolution of singularities (in the sense of \cite[2.1.7]{Jin24}), then the result holds without inverting $p$.
Compared to \cite{Jin16}, our result holds in mixed characteristic. This is because $K$-theory and $G$-theory have better functoriality in the mixed characteristic case.

Here is the organization of the paper. In Section~\ref{sec:funcG} we recall the functoriality of algebraic $G$-theory, which is used in the construction of $\KM(S)$ in Definition~\ref{def:K0M}. In Section~\ref{sec:spGys}, we recall the construction of specialization maps and Gysin morphisms in $G$-theory, and give explicit computations at the level of $G_0$. In Section~\ref{sec:funcKBM} we define the Borel-Moore $K$-motive in $\DK$ and study its functorialities; in Section~\ref{sec:DKBM} we show that over a regular base scheme, the associated Borel-Moore theory is $G$-theory (Theorem~\ref{thm:Grep}), and compare the functorialities on both sides. In Section~\ref{sec:Chowwt} we use all the results above to prove the main theorem.

\section*{Acknowledgments}

The authors would like to thank Denis-Charles Cisinski, Fr\'ed\'eric D\'eglise and Qizheng Yin for helpful discussions. The authors are supported by the National Key Research and Development Program of China Grant Nr.2021YFA1001400 and the National Natural Science Foundation of China Grant Nr.12471014.

\section{Preliminaries}

Throughout the paper, all schemes are assumed noetherian, quasi-excellent of finite Krull dimension. All morphisms of schemes are assumed separated of finite type. A morphism $X\to Y$ is \textbf{projective} if it factors as $X\xrightarrow{i}\mathbb{P}^n_Y\xrightarrow{p}Y$ with $i$ a closed immersion and $p$ the canonical projection. A morphism is \textbf{quasi-projective} if it factors as an open immersion followed by a projective morphism.
A morphism is \textbf{local complete intersection} or \textbf{lci} if it factors as a regular closed immersion followed by a smooth morphism, see \cite[\S6.6]{Ful}.
If $X$ is a scheme, we denote by $\operatorname{char}(X)$ the set of all primes $p=\operatorname{char}(k(x))$ for some point $x\in X$. For example, this is an empty set if and only if $X$ has characteristic $0$.

%Let us recall some standard definitions from \cite{BoL16}.  We say that $X$ is a \emph{retract} of $Y$ if there exist morphisms $f\colon X\to Y$ and $g\colon Y\to X$ such that $g\circ f=\id_X$. In an additive category, $X$ is a retract of $Y$ if and only if $X$ is a direct summand of $Y$.

%For categories $\mathcal C$ and $\mathcal D$, we write $\mathcal D\subset\mathcal C$ if $\mathcal D$ is a full subcategory. For a full subcategory $\mathcal D\subset\mathcal C$, we say that $\mathcal D$ is \emph{Karoubi-closed} (or \emph{idempotent complete}) in $\mathcal C$ if it contains all retracts of its objects in $\mathcal C$. The smallest Karoubi-closed subcategory of $\mathcal C$ containing $\mathcal D$ is the \emph{Karoubi closure} of $\mathcal D$ in $\mathcal C$.

%The \emph{Karoubi envelope} (or \emph{idempotent completion}) $\Kar(\mathcal D)$ of an additive category $\mathcal D$ is defined as follows: its objects are pairs $(X,e)$, where $X\in\mathcal D$ and $e\in\mathcal D(X,X)$ is an idempotent; morphisms from $(X,e)$ to $(Y,e')$ are morphisms $f\colon X\to Y$ in $\mathcal D$ such that $f=e'\circ f\circ e$.

For a stable $\infty$-category $\mathcal C$ and objects $X,Y\in\mathcal C$, we denote by $\operatorname{Map}_{\mathcal C}(X,Y)$ the mapping spectrum. We write $X\perp Y$ if $\operatorname{Map}_{\mathcal C}(X,Y)=*$. An object $M\in\mathcal C$ is called \emph{compact} if the functor $\operatorname{Map}_{\mathcal C}(M,-)$ commutes with all small coproducts. Throughout the paper, we say a diagram in an $\infty$-category is commutative if it commutes in the homotopy category, which means that it commutes up to homotopy.

%For a scheme $S$, define
%as the $\infty$-category of modules over the $\mathbb{E}_\infty$-ring spectrum $\KGL_S$.  and $\one_S(n)$ the Tate objects (\cite[Def. 2.4.17]{CD}).

%Let $\mathcal C$ be additive. For $X,Y\in\Obj\mathcal C$,  $\mathcal C(X,Y)=0$.  %(only small coproducts are considered here).

\begin{lemma}[See also \textrm{\cite[Exp. VII, \S2]{SGA6}}]
\label{lm:Tortens}
Let $A$ be a ring and let $B,C$ be two $A$-algebras. Let $M$ be a $B$-module and let $N$ be a $C$-module. 
\begin{enumerate} 
\item
For any $q\geqslant0$, the $A$-module
 %   $A ,B,C$ are commutative rings,moreover we have ring morphisms from A to B nad C,M is a B-module,N is an C-module, %then 
$\operatorname{Tor}_q^{A}(M,N)$ has a natrual $B\otimes_A C$-module structure. 

\item
If $A$ is regular, then there exists $m>0$ such that for any $q>m$ and any $M,N$ we have $\operatorname{Tor}_q^{A}(M,N)=0$.

\item
If $A$ is regular, $B,C$ are finitely generated $A$-algebras, and $M,N$ are finitely generated modules over $B$ and $C$ respectively, then $\operatorname{Tor}_q^{A}(M,N)$ is a finitely generated $B\otimes_A C$-module. %, and vanishes fo

\end{enumerate}

\end{lemma}
\proof
Choose a projective (or flat) resolution $P_\bullet \to M$ of $M$ as an $A$-module. 
Since $M$ is a $B$-module, for every $b\in B$, the multiplication by $b$
\begin{align}
\label{eq:bmult}
\begin{split}
M&\to M\\
m&\mapsto bm.
\end{split}
\end{align}
is an $A$-linear endomorphism of $M$. Since $P_\bullet$ is a projective resolution of $M$, the endomorphism~\eqref{eq:bmult} lifts to a chain map
\begin{align}
b_\bullet:P_\bullet\to P_\bullet
\end{align}
which is unique up to chain homotopy. It follows that $B$ and hence $B\otimes_A C$ act on the homology groups
\begin{align}
H_i(P_\bullet\otimes_A N)=\operatorname{Tor}_i^A(M,N).
\end{align}
The last two claims follow from the fact that if $A$ is regular, we may choose $P_\bullet$ a finite projective resolution.
%\[
%m\longmapsto bm
%\]
%defines an . Because 
\endproof

\begin{lemma}
\label{lm:Torr}
Let $A$ be a ring and let $B$ be an $A$-algebra. Let $M,N$ be $B$-modules. There is a canon ical isomorphism
\begin{align}
\label{eq:Torr}
\operatorname{Tor}_q^{B\otimes_AB}(M\otimes_AN,B)
\simeq
\operatorname{Tor}_q^{B}(M,N).
\end{align}
\end{lemma}
\proof
Choose projective (or flat) resolutions $P_\bullet \to M$ and $Q_\bullet \to N$ of $M$ and $N$ as $B$-modules. Then $R_\bullet=\operatorname{Tot}(P_\bullet\otimes_AQ_\bullet)$ is a resolution of $M\otimes_AN$, and we have
\begin{align}
\operatorname{Tor}_q^{B\otimes_AB}(M\otimes_AN,B)
=
H_i(R_\bullet\otimes_{B\otimes_AB}B)
=
H_i(\operatorname{Tot}(P_\bullet\otimes_BQ_\bullet))
=
\operatorname{Tor}_q^{B}(M,N)
\end{align}
and the result follows.
\endproof

\section{Functoriality of algebraic $G$-theory}
\label{sec:funcG}
\subsection{Basic functorialities}
\begin{enumerate}
\item
For a scheme $X$, let $G(X)$ be the \textbf{algebraic $G$-theory spectrum} of $X$ (\cite[Def. 3.3]{TT}); if $Z$ is a closed subscheme of $X$, let $K(X\textrm{ on } Z)$ be the (non-connective) relative the algebraic $K$-theory spectrum, and $K(X)=K(X\textrm{ on } X)$.

\item
The algebraic $G$-theory groups are 
\begin{align}
G_i(X)=\pi_i(G(X)).
\end{align}
The group $G_0(X)$ agrees with the Grothendieck group of coherent sheaves on $X$.

\item
The spectrum $K(X)$ has a structure of $\mathbb{E}_{\infty}$-ring spectrum; the spectrum $G(X)$ has a structure of module over $K(X)$, with a map $K(X)\wedge G(X)\to G(X)$.

\item (\cite[3.21]{TT})
There is a canonical map
\begin{align}
\label{eq:KtoG}
K(X)\to G(X)
\end{align}
which is an equivalence when $X$ is regular. More generally, if $X$ is a regular scheme and $Z$ is a closed subscheme of $X$, we have an equivalence
\begin{align}
\label{eq:KsuppG}
K(X\textrm{ on } Z)\simeq G(Z).
\end{align}

\item (\cite[3.16.1]{TT})
If $f:X\to Y$ is a proper morphism, there is a push-forward map 
\begin{align}
\label{eq:Gpf}
f_*:G(X)\to G(Y).
\end{align}

\item (\cite[3.14.1]{TT})
If $f:X\to Y$ is a morphism of finite $\operatorname{Tor}$-dimension, there is a pullback map 
\begin{align}
\label{eq:Gpb}
f^*:G(Y)\to G(X).
\end{align}

\item (\cite[\S7 4.1]{Qui})
If $f$ is the projection of a vector bundle, the map~\eqref{eq:Gpb} is an equivalence.
\end{enumerate}
\begin{lemma}
\label{lm:GTor}
If $i:Z\to X$ is a regular closed immersion defined by an ideal sheaf $\mathcal{I}$, then the map~\eqref{eq:Gpb} induces the following map on the level of $G_0$:
\begin{align}
\label{eq:Gpbri}
\begin{split}
i^*:G_0(X)&\to G_0(Z)\\
[M]&\mapsto 
\sum_{q\ge 0}
(-1)^q
\Bigl[
\operatorname{Tor}^{\mathcal{O}_X}_q
\bigl(
M,
\mathcal{O}_X/\mathcal{I}
\bigr)
\Bigr].
\end{split}
\end{align}

\end{lemma}
\proof
This follows from the fact that the right-hand side of~\eqref{eq:Gpbri} computes the class $[\operatorname{L}i^*M]$, see \cite[3.14.1]{TT}.
\endproof

%There is a canonical map $K(X)\to G(X)$, which is an equivalence if $X$ is regular. 
If $X$ is regular and $Z$ is a closed subscheme of $X$, then there is a canonical equivalence 
\begin{align}
G(Z)\simeq K(X\textrm{ on } Z).
\end{align}

There are fiber sequences 
\begin{align}
K(X\textrm{ on } Z)\to K(X)\to K(X-Z)
\end{align}
\begin{align}
\label{eq:fibG}
G(Z)\to G(X)\to G(X-Z).
\end{align}

\subsection{Product on $G_0$}

In comparison with \cite[\S1.10, \S20.2]{Ful}, one should expect to define products on $G$-theory over a regular base scheme. This can be achieved using formal methods, see \cite[2.2.7]{DJK} or~\eqref{eq:DKprod} below, but we are not aware of a direct definition for higher $G$-theory in the literature. In this paper we restrict to products on the level of $G_0$, which is sufficient for our main results.
%While it is possible to define  (, although we are not aware of )

\begin{definition}
\label{def:G0prod}
Let $S$ be a regular scheme and let $X,Y$ be two $S$-schemes of finite type. We define the product on $G_0$ as
\begin{align}
\label{eq:GproductTor}
\begin{split}
%\langle-,-\rangle\colon 
G_0(X)\times G_0(Y)&\longrightarrow G_0(X\times_SY)\\
([\mathcal F],[\mathcal G])&\mapsto\sum_{i\geq0}(-1)^i\bigl[\Tor_i^{\mathcal O_S}(\mathcal F,\mathcal G)\bigr]
\end{split}
\end{align}
which is well-defined by Lemma~\ref{lm:Tortens}.
\end{definition}
If $S$ is the spectrum of a field, the map~\eqref{eq:GproductTor} is simply given by $([\mathcal F],[\mathcal G])\mapsto[\mathcal F\boxtimes_k\mathcal G]$.

\subsection{The category of $K_0$-motives}

We recall the following definition from \cite[5.4]{GiS09}:
\begin{definition}
\label{def:K0M}
Let $S$ be a noetherian scheme. The category $\KC(S)$ of \textbf{$K_0$-correspondences over $S$} is defined as follows:
\begin{itemize} 
\item
Objects of $\KC(S)$ are projective morphisms $f:X\to S$ with $X$ a regular scheme. 
\item
Given two objects $X$ and $Y$ of $\KC(S)$, morphisms from $X$ to $Y$ are given by
\begin{align}
Hom_{\KC(S)(X,Y)}:=G_0(X\times_SY).
\end{align}
Writing $\delta_{X/S}:X\to X\times_SX$ the diagonal morphism, the identity morphism is given by $\delta_{X/S,*}[\mathcal{O}_X]\in G_0(X\times_SX)$.
\item
Given three $X, Y$ and $Z$ of $\KC(S)$, the composition of morphisms is defined as the composition
\begin{align}
\label{eq:G0comp}
G_0(X\times_SY)\times G_0(Y\times_SZ)
\xrightarrow{\eqref{eq:GproductTor}}
G_0(X\times_S Y\times_S Z)
\xrightarrow{\eqref{eq:Gpf}}
G_0(X\times_SZ)
\end{align}
where
\begin{itemize}
\item
As in~\eqref{eq:GproductTor}, the first map of~\eqref{eq:G0comp} is given by
\begin{align}
\begin{split}
G_0(X\times_S Y)\times G_0(Y\times_S Z)&\xrightarrow{\eqref{eq:GproductTor}} G_0(X\times_S Y\times_S Z)\\
([\mathcal F],[\mathcal G])&\mapsto\sum_{i\geq0}(-1)^i\bigl[\Tor_i^{\mathcal O_Y}(\mathcal F,\mathcal G)\bigr]
\end{split}
\end{align}
since $X\times_S Y\times_S Z=(X\times_S Y)\times_Y(Y\times_S Z)$ is the fiber product over the regular scheme $Y$.

\item
The second map of~\eqref{eq:G0comp} is the proper push-forward along the projection $ X\times_SY\times_SZ\longrightarrow X\times_SZ$.

\end{itemize}
\end{itemize} 
If $\mathcal P$ is a set of prime numbers, we denote by $\KC(S)[\mathcal P^{-1}]$ the $\mathbb{Z}[\mathcal P^{-1}]$-linearization of $\KC(S)$. The category $\KM(S)[\mathcal P^{-1}]$ of \textbf{$K_0$-motives over $S$} is defined as the idempotent completion of $\KC(S)[\mathcal P^{-1}]$ (see \cite[Def. 2.4]{CH}, \cite[Def. 5.6]{GiS09}). 
\end{definition}

\begin{remark}
In the definition of classical Chow motives (see \cite[Def. 2.17]{CH}), one needs an extra step of formally inverting the Tate object $\mathbb{Z}(1)$; this step in unnecessary in $K_0$-motives, as unlike Chow groups, the projective bundle formula of $G_0$ groups imply that the Tate object is already invertible in $K_0$-motives, see Lemma~\ref{omega} below.
\end{remark}

\section{Specialization and Gysin morphisms}
\label{sec:spGys}
\subsection{The specialization map}
\label{num:spmap}
We now the specialization map on $G$-theory, following \cite[\S5.2, \S20.3]{Ful} and \cite[4.5.6]{DJK}. Let $i:Z\to S$ be a regular closed immersion with open complement $j:U\to S$, and assume that there is a null-homotopy $e(N_ZS)$ of the Euler class of the normal bundle of $i$ in $G(Z)$. Let $f:X\to S$ be a morphism and form the Cartesian square
\begin{align}%\label{Cart_diag}
\begin{split}
  \xymatrix@=10pt{
    X_Z \ar[r]^-{i_X} \ar[d]_-{} & X \ar[d]^-{} & X_U \ar[l]^-{} \ar[d]^-{}\\
    Z \ar[r]_-{i} & S & U. \ar[l]^-{}
  }
\end{split}
\end{align}
By~\eqref{eq:fibG} we have a fiber sequence
\begin{align}
G(X_Z)
\underset{\eqref{eq:Gpf}}{\xrightarrow{i_{X*}}}
G(X)
\xrightarrow{\eqref{eq:Gpb}}
G(X_U).
\end{align}
By excess intersection formula (\cite[Exp. VII Prop. 3.4]{SGA6}, \cite[Prop. 3.3.4]{DJK}), the composition
\begin{align}
G(X_Z)
\underset{\eqref{eq:Gpf}}{\xrightarrow{i_{X*}}}
G(X)
\underset{\eqref{eq:Gpb}}{\xrightarrow{i_{X}^*}}
G(X_Z)
\end{align}
is null-homotopic, and therefore the map $i_{Z}^*$ factors through a map
\begin{align}
\label{eq:Gcon}
G(X_U)
\to
G(X_Z).
\end{align}

\subsection{The specialization to the normal cone}

Let $i:Z\to X$ be a closed immersion defined by an ideal sheaf $\mathcal{I}$. Recall that the \textbf{deformation space} of $i$ is
\begin{align}
D_ZX
=
\operatorname{Spec}_{\mathcal{O}_X}\left(\mathbf{R}_{\mathcal{I}}\right)
=
\operatorname{Spec}_{\mathcal{O}_X}\left(\sum_{n\in\mathbb Z} \mathcal{I}^n t^{-n}\right).
\end{align}
Here $\mathbf{R}_{\mathcal{I}}=\sum_{n\in\mathbb{Z}} \mathcal{I}^n t^{-n}\subset \mathcal{O}_X[t,t^{-1}]$ is the Rees algebra, where for $n\leqslant0$ we put $\mathcal{I}^n=\mathcal{O}_X$. There is a Cartesian diagram
\begin{align}%\label{Cart_diag}
\label{eq:defdiag}
\begin{split}
  \xymatrix@=10pt{
     C_ZX \ar[r]^-{} \ar[d]_-{} & D_ZX \ar[d]^-{} & X\times\mathbb{G}_m \ar[l]^-{} \ar[d]^-{}\\
     0 \ar[r]_-{} & \mathbb{A}^1 & \mathbb{G}_m \ar[l]^-{}
  }
\end{split}
\end{align}
where $C_ZX$ is the \textbf{normal cone} of $i$:
\begin{align}
C_ZX
=\operatorname{Spec}_{\mathcal{O}_X}\left(
\bigoplus_{n\ge 0} \mathcal{I}^n/\mathcal{I}^{n+1}
\right)
=\operatorname{Spec}\left(\operatorname{gr}_{\mathcal{I}}(\mathcal{O}_X)\right).
\end{align}

%There is a canonical map $D_ZX\to\mathbb{A}^1$, where the fiber of $\mathbb{G}_m$ is $
%\operatorname{Spec}_{\mathcal{O}_X}\left(\sum_{n\in\mathbf Z} \mathcal{I}^n t^{-n}\right)[t^{-1}]\simeq\operatorname{Spec}_{\mathcal{O}_X}(\mathcal{O}_X[t,t^{-1}])=X\times\mathbf G_m$

\subsection{}
\label{num:spG}
Applying the construction in~\ref{num:spmap} to the diagram~\eqref{eq:defdiag}, we obtain a map
\begin{align}
\label{eq:Gcone}
G(X\times\mathbb G_m)
\xrightarrow{\eqref{eq:Gcon}}
G(C_ZX).
\end{align}

%Let $i_0:C_ZX\to D_ZX$ be the inclusion. By~\eqref{eq:fibG} we have a fiber sequence
%\begin{align}
%G(C_ZX)
%\underset{\eqref{eq:Gpf}}{\xrightarrow{i_{0*}}}
%G(D_ZX)
%\xrightarrow{\eqref{eq:Gpb}}
%G(X\times\mathbb G_m).
%\end{align}
%Since $i_0$ is the inclusion of a principal divisor, by excess intersection formula (\cite[Prop. 3.3.4]{DJK}), the composition
%\begin{align}
%G(C_ZX)
%\underset{\eqref{eq:Gpf}}{\xrightarrow{i_{0*}}}
%G(D_ZX)
%\underset{\eqref{eq:Gpb}}{\xrightarrow{i_{0}^*}}
%G(C_ZX)
%\end{align}
%is null-homotopic, and therefore the map $i_{0}^*$ factors through a map

\begin{definition}
\label{def:spcone}
The \textbf{specialization to the normal cone} is the composition
\begin{align}
\label{eq:spdef}
\operatorname{sp}:
G(X)
\xrightarrow{\eqref{eq:Gpb}}
G(X\times\mathbb G_m)
\xrightarrow{\eqref{eq:Gcone}}
G(C_ZX).
\end{align}
\end{definition}

\subsection{}
We now give another description, following \cite[3.2.3]{DJK}. Let %(see also \cite[Lemma 3.3.6]{Jin18}). Let 
\begin{align}
\label{eq:gamt}
\gamma_t:
G(X)
\xrightarrow{}
\Omega G(X\times\mathbb G_m)
\end{align}
be the multiplication by $t\in \Omega K(\mathbb{Z}[t,t^{-1}])$.

\begin{lemma}
\label{lm:spano}
The specialization map~\eqref{eq:spdef} is homotopic to the composition
\begin{align}
G(X)
\xrightarrow{\eqref{eq:gamt}}
\Omega G(X\times\mathbb G_m)
\xrightarrow{\partial}
G(C_ZX)
\end{align}
where $\partial$ is the boundary map 
induced by the fiber sequence~\eqref{eq:fibG}.
\end{lemma}
\proof
The result follows from the commutative diagram
\begin{align}%\label{Cart_diag}
\label{eq:Jin45}
\begin{split}
  \xymatrix@=10pt{
    G(X\times\mathbb{G}_m)  \ar[d]_-{t} & G(D_ZX) \ar[l]_-{\eqref{eq:Gpb}} \ar[d]^-{i_0^*} \\
    \Omega G(X\times\mathbb G_m) \ar[r]^-{\partial} & G(C_ZX).
  }
\end{split}
\end{align}
associated to the inclusion of a principal divisor $i_0:C_ZX\to D_ZX$ and the open complement $X\times\mathbb{G}_m$. To prove the commutativity of diagram~\eqref{eq:Jin45}, one may apply the argument in \cite[Prop. 2.30]{Jin16}: indeed, using the double deformation to the normal cone, it suffices to prove the commutativity of the diagram
\begin{align}%\label{Cart_diag}
\label{eq:DJK322}
\begin{split}
  \xymatrix@=10pt{
    G(C_ZX\times\mathbb{G}_m)  \ar[d]_-{t} & G(C_ZX\times\mathbb{A}^1) \ar[l]_-{\eqref{eq:Gpb}} \ar[d]^-{i_0^*} \\
    \Omega G(C_ZX\times\mathbb G_m) \ar[r]^-{\partial} & G(C_ZX).
  }
\end{split}
\end{align}
which is the counterpart of diagram~\eqref{eq:Jin45} with the inclusion $i_0$ replaced by the zero section $C_ZX\to C_ZX\times\mathbb{A}^1$. Diagram~\eqref{eq:DJK322} commutes by \cite[3.2.2]{DJK}, which finishes the proof.
% we may replace  by the zero section of  given by 
\endproof

\subsection{}
We now give a concrete description at the level of $G_0$. In this case, the map $G_0(D_ZX)\to G_0(X\times\mathbb G_m)$ is surjective, so the map 
\begin{align}
G_0(X\times\mathbb G_m)\to G_0(C_ZX)
\end{align}
induced by~\eqref{eq:Gcone} is simply the map induced by $i_{0}^*:G_0(D_ZX)\xrightarrow{}G_0(C_ZX)$ on the cokernel $G_0(X\times\mathbb G_m)=\operatorname{coker}(G_0(C_ZX)\xrightarrow{i_{0*}}G_0(D_ZX))$, see \cite[Prop. 5.2]{Ful}.

Let $M$ be a coherent sheaf on $X$, with its class $[M]\in G_0(X)$. The map $G_0(X)\xrightarrow{\eqref{eq:Gpb}}G_0(X\times\mathbb G_m)$ sends $[M]$ to $[M[t,t^{-1}]]\in G_0(X\times\mathbb G_m)$, and the class of the Rees module
\begin{align}
\mathbf{R}_{\mathcal{I}}(M)
=
\sum_{n\in\mathbb{Z}}
\mathcal{I}^nMt^{-n}
\subset M[t,t^{-1}]
\end{align}
is a lift of $[M[t,t^{-1}]]$ to $G_0(D_ZX)$. By Lemma~\ref{lm:GTor}, it follows that
\begin{align}
\operatorname{sp}[M]
=
i_0^*[\mathbf{R}_{\mathcal{I}}(M)]
=
\sum_{q\ge 0}
(-1)^q
\Bigl[
\operatorname{Tor}^{\mathbf{R}_{\mathcal{I}}}_q
\bigl(
\mathbf{R}_{\mathcal{I}}(M),
\mathbf{R}_{\mathcal{I}}/t
\bigr)
\Bigr].
\end{align}
Using the resolution
\begin{align}
\mathbf{R}_{\mathcal{I}} 
\xrightarrow{\times t} 
\mathbf{R}_{\mathcal{I}} 
\xrightarrow{}
\mathbf{R}_{\mathcal{I}}/t
\end{align}
we obtain that for any $q\geqslant1$, $\operatorname{Tor}^{\mathbf{R}_{\mathcal{I}}}_q
\bigl(
\mathbf{R}_{\mathcal{I}}(M),
\mathbf{R}_{\mathcal{I}}/t
\bigr)=0$, and
\begin{align}
\operatorname{Tor}^{\mathbf{R}_{\mathcal{I}}}_0
\bigl(
\mathbf{R}_{\mathcal{I}}(M),
\mathbf{R}_{\mathcal{I}}/t
\bigr)
=
\operatorname{gr}_{\mathcal{I}}(M)
=
\bigoplus_{n\ge 0}
\mathcal{I}^nM/\mathcal{I}^{n+1}M.
\end{align}
We therefore obtain:
\begin{lemma}
\label{lm:spexp}
The specialization map $\operatorname{sp}:
G_0(X)
\xrightarrow{\eqref{eq:spdef}}
G_0(C_ZX)$ is given by
\begin{align}
\operatorname{sp}[M]
=
[\operatorname{gr}_\mathcal{I}(M)]
=
\Bigl[\bigoplus_{n\ge 0}
\mathcal{I}^nM/\mathcal{I}^{n+1}M\Bigl].
\end{align}
\end{lemma}

%As an interesting consequence we obtain the following additivity property of the associated graded module in $G$-theory
%From  we deduce:
\begin{corollary}
\label{cor:grG}
Let
$
0\to M_1
\to M_2
\to M_3
\to 0
$
be a short exact sequence of $\mathcal{O}_X$-modules. Then in $G_0(C_ZX)=G_0(\operatorname{gr}_{\mathcal{I}}(\mathcal{O}_X))$ we have
\begin{align}
[\operatorname{gr}_{\mathcal{I}}(M_2)]
=
[\operatorname{gr}_{\mathcal{I}}(M_1)]
+
[\operatorname{gr}_{\mathcal{I}}(M_3)]
\end{align}
\end{corollary}
Corollary~\ref{cor:grG} is deduced from Lemma~\ref{lm:spexp} and is non-trivial, as the functor $gr_{\mathcal{I}}(-)$ is not exact.

\subsection{The Gysin morphism}

Let $i:Z\to X$ be a regular closed immersion defined by an ideal sheaf $\mathcal{I}$. Then the normal cone $C_ZX$ is canonically identified with the normal bundle $N_ZX$, with $\pi:N_ZX\to Z$ the projection. Following \cite[Example 5.2.1]{Ful} and \cite[Def. 2.18]{Jin16}, we define the \textbf{Gysin morphism} as the composition
\begin{align}
\label{eq:Gys}
G(X)
\underset{\eqref{eq:spdef}}{\xrightarrow{\operatorname{sp}}}
G(N_ZX)
\xrightarrow{(\pi^*)^{-1}}
G(Z).
\end{align}

\begin{lemma}%[See also \textrm{\cite[Lemma 3.3.6]{Jin18}}]
\label{lm:Gysinsp}
The Gysin morphism~\eqref{eq:Gys} is homotopic to the pullback map $i^*:G(X)\to G(Z)$ in~\eqref{eq:Gpb}.
\end{lemma}
\proof
%By Lemma~\ref{lm:spano}
The deformation space $D_ZX$ is an open subscheme of the blowup $\operatorname{Bl}_{Z\times0}X\times\mathbb{A}^1$ at a regularly embedded closed subscheme, and therefore there is a canonical morphism $r:D_ZX\to X\times\mathbb{A}^1$ which is lci by \cite[Tag 0BIQ]{Stack}. The morphism $r$ induces a Cartesian diagram
\begin{align}%\label{Cart_diag}
\begin{split}
  \xymatrix@=10pt{
    N_ZX \ar[r]^-{} \ar[d]_-{} & D_ZX \ar[d]^-{} & X\times\mathbb{G}_m \ar[l]^-{} \ar@{=}[d]^-{}\\
     X \ar[r]_-{} & X\times\mathbb{A}^1 & X\times\mathbb{G}_m \ar[l]^-{}
  }
\end{split}
\end{align}
where the left vertical map is the composition $N_ZX\xrightarrow{\pi}Z\xrightarrow{i}X$, see \cite[3.2.3]{DJK}. In particular we have a commutative diagram
\begin{align}%\label{Cart_diag}
%\label{eq:Jin45}
\begin{split}
  \xymatrix@=10pt{
    G(X\times\mathbb{G}_m) \ar@/^1pc/@{.>}[rr]^-{\eqref{eq:Gcone}} & G(D_ZX) \ar[l]_-{} \ar[r]_-{i_0^*} & G(N_ZX) \\
     & G(X) \ar[r]^-{i^*} \ar[lu]^-{\eqref{eq:Gpb}} \ar[u]_-{r^*} & G(Z)  \ar[u]_-{\pi^*}.
  }
\end{split}
\end{align}
and the result follows from Definition~\ref{def:spcone}.
\endproof

\begin{remark}
We can explicitly check the validity of Lemma~\ref{lm:Gysinsp} at the level of $G_0$ .% of the Gysin morphism $i^*:G_0(X)\to G_0(Z)$ in~\eqref{eq:Gys}. 
Working locally, we assume that $X=\operatorname{Spec}(A)$ is affine, and $\mathcal{I}=(f_1,\dots,f_r)$ is generated by a regular sequence. Then $i:Z\to X$ factors as a composition of codimension-one regular immersions
\begin{align}
Z=Z_r\to Z_{r-1}\to\cdots\to Z_0=X
\end{align}
with $Z_j=\operatorname{Spec}(A/(f_1,\dots,f_j))$. So it suffices to consider the case where $\mathcal{I}=(f)$ is principal. Let $M$ be a finitely generated $A$-module. Then the pullback map~\eqref{eq:Gpb} gives
\begin{align}
i^*[M]
=
[M\xrightarrow{\times f}M]\in G_0(A/(f)).
\end{align}
On the other hand, the Gysin morphism~\eqref{eq:Gys} sends the class $[M]$ to the class of the Koszul complex
\begin{align}
K\!\left(
\bar f,
\bigoplus_{n\ge0}f^nM/f^{n+1}M
\right)
=
\left(\bigoplus_{n\ge0}
f^nM/f^{n+1}M
\xrightarrow{\times \bar f}
\bigoplus_{n\ge0}
f^nM/f^{n+1}M\right).
\end{align}
where $\bar f$ denotes the image of $f$ in $I/I^2$. To see that these two complexes have the same class in $G_0(A/(f))$, since they both have $H_0$ isomorphic to $M/fM$, it suffices to compare $H_1$ of both complexes. We distinguish the following cases:
\begin{enumerate}
\item[\textbf{Case 1.}]
If the multiplication by $f$ is injective on $M$, then the $H_1$ of both complexes vanish.
\item[\textbf{Case 2.}]
Assume that $f^nM=0$ for some $n\ge1$. Then in $G_0(A)$ we have
\begin{align}
[M]
=
[M/fM]
+
[fM/f^2M]
+\cdots+
[f^{n-1}M].
\end{align}
and from Corollary~\ref{cor:grG} we obtain
\begin{align}
[\operatorname{gr}_I(M)]
=
[\operatorname{gr}_I(M/fM)]
+\cdots+
[\operatorname{gr}_I(f^{n-1}M)].
\end{align}
Therefore it suffices to consider the case where $fM=0$, in which case both complexes have $H_1$ isomorphic to $M$.
\item[\textbf{General case.}]
In general, denote by 
\begin{align}
M[f^\infty]
=
\{m\in M\mid f^n m=0
\text{ for some }n\ge1\}
\end{align}
the $f$-power torsion submodule. Then in $G_0(A)$,we have $[M]=[M[f^\infty]]+[M/M[f^\infty]]$. The module $M[f^\infty]$ falls under Case~2, while $M/M[f^\infty]$ falls under Case~1. By Corollary~\ref{cor:grG}, the statement follows for $M$.
\end{enumerate}
\end{remark}

\begin{corollary}\label{two_koszul}
Assume that $\mathcal{I}=(f_1,\dots,f_r)$ is generated by a regular sequence, and denote by $\bar{f_i}$ the image of $f_i$ in $\mathcal{I}/\mathcal{I}^2$. Then the two Koszul complexes $K(f_1,...,f_r,M)$ and $K(\bar{f_1},...,\bar{f_r},gr_{\mathcal{I}}(M))$ have the same class in $G_0(Z)$.
\end{corollary}

\subsection{The Refined Gysin Morphism}

Let $i:Z\to X$ be a regular closed immersion defined by an ideal sheaf $\mathcal{I}$, and consider a Cartesian square
\begin{align}%\label{Cart_diag}
\begin{split}
  \xymatrix@=10pt{
    W \ar[r]^-{} \ar[d]_-{} & Y \ar[d]^-{f} \\ %\ar@{}[rd]|{\Delta} 
    Z \ar[r]^-{i} & X.
  }
\end{split}
\end{align}
Then we have $W\simeq\operatorname{Spec}_{\mathcal{O}_Y}(\mathcal{O}_Y/\mathcal{I})$.
Let $N'=N_ZX\times_XY$ and denote by $k:C_WY\to N'$ the canonical closed immersion (\cite[B.6.1]{Ful}) and $\pi:N'\to W$ the projection. Following \cite[\S6.2]{Ful} and \cite[Def. 2.31]{Jin16}, we define
\begin{definition}
\label{def:refGys}
The \textbf{refined Gysin morphism} is the composition
\begin{align}
\label{eq:refGys}
\operatorname{G}_{f}(i):
G(Y)
\underset{\eqref{eq:spdef}}{\xrightarrow{\operatorname{sp}}}
G(C_WY)
\underset{\eqref{eq:Gpf}}{\xrightarrow{k_*}}
G(N')
\xrightarrow{(\pi^*)^{-1}}
G(W).
\end{align}
\end{definition}

\subsection{}
It is straightforward to generalize the construction to the case where $i$ is lci, see \cite[\S6.6]{Ful}. %, which will not be needed in what follows.
We now give an alternative construction following \cite[Cor. 2.34]{Jin16}. Let $D'=D_ZX\times_XY$. We have a %$\operatorname{Tor}$-independent 
Cartesian diagram
\begin{align}\label{eq:CND}
\begin{split}
  \xymatrix@=10pt{
    C_WY \ar[r]^-{\alpha} \ar[d]_-{k} & D_WY \ar[d]^-{\gamma} & Y\times\mathbb{G}_m \ar@{=}[d]^-{} \ar[l]^-{} \\
    N' \ar[r]_-{\beta} \ar[d]^-{} & D' \ar[d]^-{} & Y\times\mathbb{G}_m \ar[l]^-{} \ar[d]^-{}\\
     0 \ar[r]_-{} & \mathbb{A}^1 & \mathbb{G}_m \ar[l]^-{}
  }
\end{split}
\end{align}
where $N'\to D'$ is also the inclusion of a principal divisor. Applying the construction in~\ref{num:spmap} to the lower part of diagram~\eqref{eq:CND}, % and proceeding as in~\ref{num:spG}, 
we obtain a map %similar to the map~
\begin{align}
\label{eq:Gcone2}
G(Y\times\mathbb G_m)
\xrightarrow{\eqref{eq:Gcon}}
G(N').
\end{align}
\begin{lemma}
The map~\eqref{eq:Gcone2} is homotopic to the composition
\begin{align}
G(Y\times\mathbb G_m)
\xrightarrow{\eqref{eq:Gcone}}
G(C_WY)
\xrightarrow{k_*}
G(N').
\end{align}
\end{lemma}
\proof
Tracking back the definitions, the result follows from the following commutative diagram:
\begin{align}
\begin{split}
  \xymatrix@=10pt{
    G(D_WY) \ar[r]^-{\alpha^*} \ar[d]_-{\gamma_*} & G(C_WY) \ar[d]^-{k_*} \\
    G(D') \ar[r]^-{\beta^*} & G(N')
  }
\end{split}
\end{align}
which is \cite[3.18]{TT}.
\endproof
\begin{corollary}
\label{cor:spnew}
Denoting by
\begin{align}
\label{eq:spnew}
\operatorname{sp}':
G(Y)
\xrightarrow{\eqref{eq:Gpb}}
G(Y\times\mathbb G_m)
\xrightarrow{\eqref{eq:Gcone2}}
G(N').
\end{align}
the analogue of the specialization map~\eqref{eq:spdef}, the refined Gysin morphism~\eqref{eq:refGys} is homotopic to the compositon
\begin{align}
G(Y)
\underset{\eqref{eq:spnew}}{\xrightarrow{\operatorname{sp}'}}
G(N')
\xrightarrow{(\pi^*)^{-1}}
G(W).
\end{align}
\end{corollary}

\subsection{}
We now give an explicit description of the refined Gysin morphism~\eqref{eq:refGys} on the level of $G_0$. If $f:X\to Y$ is a morphism and $M$ is a quasi-coherent sheaf on $Y$, we denote $M_{|X}=\mathbf Lf^*M$.
\begin{lemma}
\label{lm:refGysexp}
The refined Gysin morphism $\operatorname{G}_{f}(i):
G_0(Y)
\xrightarrow{\eqref{eq:refGys}}
G_0(W)$ sends the class of a coherent sheaf $M$ on $Y$ to
\begin{align}
\label{eq:refTor}
\operatorname{G}_{f}(i)[M]
=
[M\otimes^{\mathbf L}_{\mathcal{O}_X}\mathcal{O}_Z]
=
\sum_{q\ge0}
(-1)^q
\Bigl[
\operatorname{Tor}^{\mathcal{O}_X}_q(M,\mathcal{O}_X/\mathcal{I})
\Bigr]
\in G_0(W).
\end{align}
Here $\operatorname{Tor}^{\mathcal{O}_X}_q(M,\mathcal{O}_X/\mathcal{I})$ is a $\mathcal{O}_W$-module by Lemma~\ref{lm:Tortens}.
\end{lemma}

\proof
In the lower row of~\eqref{eq:CND}, let $M'$ be any coherent sheaf on $D'$ extending the sheaf $M[t,t^{-1}]$ on $Y\times\mathbb{G}_m$. For example, one may choose $M'=M\otimes_{\mathcal{O}_Y}\mathcal{O}_{D'}$. Then the map $\operatorname{sp}':G_0(Y)\to G_0(N')$ induced by~\eqref{eq:spnew} is such that
\begin{align}
\label{eq:splift}
\operatorname{sp}'[M]=[M'_{|N'}].
\end{align}
We have a Cartesian diagram
\begin{align}
\begin{split}
  \xymatrix@=10pt{
    W\times\mathbb{A}^1 \ar[r]^-{} \ar[d]_-{} & D' \ar[d]^-{} \\
    Z\times\mathbb{A}^1 \ar[r]^-{l} & D_ZX
  }
\end{split}
\end{align}
and let
\begin{align}
\Gamma
=\left(
(l_*\mathcal{O}_{Z\times\mathbb{A}^1})_{|D'}\otimes^{\mathbf L}_{\mathcal{O}_{D'}}M'
\right)_{|W\times\mathbb{A}^1}
\end{align}
as a sheaf on $W\times\mathbb{A}^1$. By $\mathbb{A}^1$-invariance of $G$-theory, the two sections $s_0,s_1:W\to W\times\mathbb{A}^1$ are such that $s_0^*[\Gamma]=s_1^*[\Gamma]\in G_0(W)$. We now compute $\mathbf Ls_0^*\Gamma$ and $\mathbf Ls_1^*\Gamma$ explicitly.

For $\mathbf Ls_0^*\Gamma$, we have a %$\operatorname{Tor}$-independent 
Cartesian diagram
\begin{align}%\label{Cart_diag}
\begin{split}
  \xymatrix@=10pt{
    W \ar[r]^-{\sigma} \ar[d]_-{s_0} & N' \ar[d]^-{} \\
    W\times\mathbb{A}^1 \ar[r]_-{} & D'
  }
\end{split}
\end{align}
where $\sigma$ is the zero section of $N'$. Therefore 
\begin{align}
\label{eq:Ls0a}
\begin{split}
\mathbf Ls_0^*\Gamma
=
\left(\left(
(l_*\mathcal{O}_{Z\times\mathbb{A}^1})_{|D'}\otimes^{\mathbf L}_{\mathcal{O}_{D'}}M'
\right)_{|N'}\right)_{|W}
=
\left(
(l_*\mathcal{O}_{Z\times\mathbb{A}^1})_{|N'}\otimes^{\mathbf L}_{\mathcal{O}_{N'}}M'_{|N'}
\right)_{|W}.
\end{split}
\end{align}
We also have a $\operatorname{Tor}$-independent Cartesian diagram
\begin{align}%\label{Cart_diag}
\begin{split}
  \xymatrix@=10pt{
      W \ar[d]_-{} \ar[r]^-{\sigma} & N' \ar[d]^-{} \\
      Z \ar[r]^-{} \ar[d]^-{} & N_ZX \ar[d]_-{}\\
      Z\times\mathbb{A}^1 \ar[r]^-{l} & D_ZX
  }
\end{split}
\end{align}
and by $\operatorname{Tor}$-independent base change we have
\begin{align}
\label{eq:Ls01}
(l_*\mathcal{O}_{Z\times\mathbb{A}^1})_{|N'}
\simeq
\sigma_*\mathcal{O}_{W}
\end{align}
and therefore combining~\eqref{eq:Ls0a} and~\eqref{eq:Ls01} we obtain
\begin{align}
\label{eq:Ls0b}
\mathbf Ls_0^*\Gamma
\simeq
\left(
\sigma_*\mathcal{O}_{W}\otimes^{\mathbf L}_{\mathcal{O}_{N'}}M'_{|N'}
\right)_{|W}
\simeq
(M'_{|N'})_{|W}.
\end{align}
By Corollary~\ref{cor:spnew},~\eqref{eq:splift} and~\eqref{eq:Ls0b}, we obtain
\begin{align}
\label{eq:s0rG}
s_0^*[\Gamma]
=
\operatorname{G}_{\Delta}(i)[M].
\end{align}

For $\mathbf Ls_1^*\Gamma$, we have a Cartesian diagram
\begin{align}%\label{Cart_diag}
\begin{split}
  \xymatrix@=10pt{
      W \ar[d]_-{\sigma_1} \ar[r]^-{} & Y \ar[d]^-{} \\
      W\times\mathbb{G}_m \ar[r]^-{} \ar[d]^-{} & Y\times\mathbb{G}_m \ar[d]_-{}\\
      W\times\mathbb{A}^1 \ar[r]^-{} & D'
  }
\end{split}
\end{align}
and therefore 
\begin{align}
\label{eq:Ls1a}
\begin{split}
\mathbf Ls_1^*\Gamma
=
\mathbf L\sigma_1^*\left(
(l_*\mathcal{O}_{Z\times\mathbb{A}^1})_{|D'}\otimes^{\mathbf L}_{\mathcal{O}_{D'}}M'
\right)_{|W\times\mathbb{G}_m}
%&=
%\mathbf L\sigma_1^*\left(
%(l_*\mathcal{O}_{Z\times\mathbb{A}^1})_{|Y\times\mathbb{G}_m}\otimes^{\mathbf L}_{\mathcal{O}_{Y\times\mathbb{G}_m}}M_{|Y\times\mathbb{G}_m}
%\right)_{|W\times\mathbb{G}_m}\\
%&
=
\left(
(l_*\mathcal{O}_{Z\times\mathbb{A}^1})_{|Y}\otimes^{\mathbf L}_{\mathcal{O}_{Y}}M
\right)_{|W}.
\end{split}
\end{align}
We also have a $\operatorname{Tor}$-independent Cartesian diagram
\begin{align}%\label{Cart_diag}
\begin{split}
  \xymatrix@=10pt{
      Z \ar[d]_-{} \ar[r]^-{i} & X \ar[d]^-{} \\
      Z\times\mathbb{G}_m \ar[r]^-{} \ar[d]^-{} & X\times\mathbb{G}_m \ar[d]_-{}\\
      Z\times\mathbb{A}^1 \ar[r]^-{l} & D_ZX
  }
\end{split}
\end{align}
and by $\operatorname{Tor}$-independent base change we obtain
\begin{align}
\label{eq:Ls1b}
(l_*\mathcal{O}_{Z\times\mathbb{A}^1})_{|Y}
=
(l_*\mathcal{O}_{Z\times\mathbb{A}^1})_{|X})_{|Y}
\simeq
(i_*\mathcal{O}_{Z})_{|Y}.
\end{align}
Combining~\eqref{eq:Ls1a} and~\eqref{eq:Ls1b}, we obtain
\begin{align}
\label{eq:s1ZL}
s_1^*[\Gamma]
=\Bigl[\left(
(i_*\mathcal{O}_{Z})_{|Y}\otimes^{\mathbf L}_{\mathcal{O}_{Y}}M
\right)_{|W}\Bigr]
=
[\mathcal{O}_Z\otimes^{\mathbf L}_{\mathcal{O}_{X}}M].
\end{align}
We conclude by combining~\eqref{eq:s0rG} and~\eqref{eq:s1ZL}.
\endproof

\section{Functoriality of Borel--Moore $K$-motives}
\label{sec:funcKBM}
We discuss basic properties of Borel--Moore $K$-motives in $\DK$ via the abstract six functors formalism (\cite{CD}).  We follow \cite{Jin16}, where results are stated in the category of cdh-motives $\mathbf{DM}_{cdh}$, but the main results remain valid in $\DK$. In addition, a remarkable feature of algebraic $K$-theory is the Bott Periodicity, which we describe now.

\subsection{Periodicity in $\DK$}
\label{num:regG}
Recall that for a scheme $S$, $\DK(S)$ is the $\infty$-category of modules over the $\mathbb{E}_\infty$-ring spectrum $\KGL_S$. We denote by $\one_S\in\DK(S)$ the unit object. By \cite[Prop. 2.18]{Cis13}, for any scheme $S$ and any smooth $S$-scheme $X$, there is an equivalence
\begin{align}
\label{eq:DKones}
\operatorname{Map}_{\DK(S)}(X_+,\one_S)
\simeq
KH(X).
\end{align}
\subsection{}
We will use the following periodicity of algebraic $K$-theory spectrum.
Let %$X$ be a scheme and let 
$\mathcal{E}$ be a vector bundle of rank $r$ over $S$. Recall that by \cite[Proposition 3.2.17]{MV} we have an isomorphism in $\DK(S)$
\begin{equation}
\label{Thom_proj_bundle}
\operatorname{Th}(\mathcal{E})
\simeq
\mathbb{P}(\mathcal{E}\oplus\mathcal{O}_S)/\mathbb{P}(\mathcal{E}).
\end{equation}
Let $x$ be the point of the $K$-theory spectrum $K(\mathbb{P}(\mathcal{E}\oplus\mathcal{O}_S))$ corresponding to the sheaf $\mathcal{O}(1)$, and denote by $\nu(\mathcal{E})$ the class of
\begin{equation}
\label{gen_bott_class}
\nu(\mathcal{E})=x^r-[\wedge^1\mathcal{E}]x^{r-1}+\cdots+(-1)^r[\wedge^r\mathcal{E}]
\end{equation}
considered as a point of the $K$-theory spectrum $K(\mathbb{P}(\mathcal{E}\oplus\mathcal{O}_S))$. Since the restriction of $\nu(\mathcal{E})$ to $\mathbb{P}(\mathcal{E})$ is null-homotopic, by~\eqref{eq:DKones}, the element $\nu(\mathcal{E})$ induces a map in $\DK(S)$
\begin{equation}
\label{eq:perK}
\operatorname{Th}(\mathcal{E})\to \one_S.
\end{equation}
\begin{lemma}[see also \textrm{\cite[Lemma 6.1.3.3]{Rio10}}]
\label{omega}
The map~\eqref{eq:perK} is an isomorphism.
\end{lemma}
\proof
For any $X\in Sm/S$, the map~\eqref{eq:perK} induces a map
\begin{align}
\begin{split}
\label{omega_spec_map}
KH(X)
&\overset{\eqref{eq:DKones}}{\simeq}
\operatorname{Map}_{\DK(S)}(X_+,\one_S)
\to
\operatorname{Map}_{\DK(S)}(\operatorname{Th}(\mathcal{E})\wedge X_+,\one_S)\\
&\overset{\eqref{Thom_proj_bundle}}{\simeq}
\operatorname{fib}\left(
KH(\mathbb{P}(\mathcal{E}\oplus\mathcal{O}_S)\times_SX)
\to 
KH(\mathbb{P}(\mathcal{E})\times_SX)\right).
\end{split}
\end{align}
The induced map $KH(X)\to KH(\mathbb{P}(\mathcal{E}\oplus\mathcal{O}_S)\times_SX)$ is given by $a\mapsto \nu(\mathcal{E})\cdot p^*a$. Therefore the map~\eqref{omega_spec_map} is an isomorphism by the projective bundle formula for $KH$-theory (\cite[Th. 2.2.1.13]{Deg11}), and the result follows.
\endproof

\begin{corollary}[See \textrm{\cite[4.3.1]{DJK}}]
If $f:X\to S$ is an lci morphism, then we have a natural tranformation in $\DK(X)$
\begin{align}
\label{eq:relpurK}
\mathfrak{p}_f:
f^*(-)
\overset{\eqref{eq:perK}}{\simeq}
f^*(-)\otimes\operatorname{Th}(T_f)
\to
f^!(-)
\end{align}
which is an isomorphism if $f$ is smooth. If $g:Y\to X$ is another lci morphism such that $fg$ is lci, then there is a canonical homotopy between the map $\mathfrak{p}_{fg}$ and the composition
\begin{align}
g^*f^*
\xrightarrow{\mathfrak{p}_g}
g^!f^*
\xrightarrow{\mathfrak{p}_f}
g^!f^!
\simeq
(fg)^!.
\end{align}
\end{corollary}

\begin{definition}
For a morphism $f\colon X\to S$, define the \textbf{Borel--Moore $K$-motive} of $X$ over $S$ as
\begin{align}
\KBM(X/S):=f_!\one_X\in\DK(S).
\end{align}
%Unless stated otherwise, morphisms discussed below are separated and of finite type.
\end{definition}

%\begin{lemma}\label{lem:BM-functoriality}

\subsection{}
We recall the following functorialities of the Borel--Moore $K$-motive. These properties are consequences of the abstract six-functors formalism, which hold in $\DK$ as well as in $\mathbf{DM}_{cdh}$, see \cite[Cor. 13.3.5]{CD}. Details can be found in \cite[Lemma 2.4, Def. 2.21, Lemma 2.23]{Jin16}, \cite[Def. 3.2.7]{Deg11} and \cite[Thm. 4.2.1, 4.5.6]{DJK}. In what follows, assume that all schemes are separated of finite type over the base scheme $S$.
%Let $p\colon Y\to S$ be separated and of finite type.
\begin{enumerate}
\item(Functorialities) Let $X\xrightarrow{f}Y\xrightarrow{p}S$ be two composable morphisms.
\begin{enumerate}
\item
If $f$ is proper, there is a map
\begin{align}
\label{eq:BMprop}
\KBM(Y/S)
=
p_!\one_Y\to p_!f_*f^*\one_S
=
\KBM(X/S).
\end{align}
\item
If $f$ is lci, there is a map
\begin{align}
\label{eq:BMsm}
\begin{split}
\KBM(X/S)
=
p_!f_!f^*\one_S
\overset{\eqref{eq:relpurK}}{\simeq}
p_!f_!f^!\one_S
\to
p_!\one_S
=
\KBM(Y/S).
\end{split}
\end{align}
\end{enumerate}

\item(Localization) For a closed immersion $i\colon Z\hookrightarrow X$ with complementary open immersion $j\colon U\hookrightarrow X$, there is a canonical map 
\begin{align}
\partial_{X,Z}\colon\KBM(Z/S)\to\KBM(U/S)[1]
\end{align}
inducing a distinguished triangle
\begin{align}
\label{eq:KBMloc}
\KBM(U/S)\xrightarrow{j_*}\KBM(X/S)\xrightarrow{i^*}\KBM(Z/S)\xrightarrow{\partial_{X,Z}}\KBM(U/S)[1].
\end{align}
If $i$ is a regular closed immersion such that the normal bundle $N_i$ is endowed with a trivialization, there is map analogous to~\eqref{eq:Gcon}
\begin{align}
\label{eq:spKBM}
\KBM(Z/S)\to \KBM(U/S).
\end{align}

\item(Base change)
Given a Cartesian square
\begin{align}
\begin{split}
  \xymatrix@=10pt{
    X' \ar[r]^-{q} \ar[d]_-{g} & Y' \ar[d]^-{f} \\
    X \ar[r]^-{p} & Y
  }
\end{split}
\end{align}
there is an isomorphism
\begin{align}
\label{eq:BMBC}
\KBM(X'/Y')=q_!g^*\one_X
\simeq
f^*p_!\one_X=f^*\KBM(X/Y).
\end{align}
\item(K\"unneth formula)
For $S$-schemes $X,Y$, there is an isomorphism
\begin{align}
\label{eq:BMKun}
\KBM(X\times_SY/S)\simeq\KBM(X/S)\otimes\KBM(Y/S).
\end{align}
%is distinguished.
\end{enumerate}

\begin{definition}[Refined construction, see \textrm{\cite[Def. 2.31]{Jin16}}]\label{def:refined-Gysin-motive}
\label{def:refDK}
Consider a Cartesian square
\begin{align}
\begin{split}
  \xymatrix@=10pt{
    X' \ar[r]^-{} \ar[d]_-{} & Y' \ar[d]^-{f} \\
    X \ar[r]^-{g} & Y
  }
\end{split}
\end{align}
where $g$ is lci, and let $r\colon Y'\to S$ be a morphism. Define a map
\begin{align}
\label{eq:rGKBM}
R_f(g)\colon\KBM(X'/S)\longrightarrow\KBM(Y'/S)
\end{align}
by applying the functor $r_!$ to the composition
\begin{align}
\KBM(X'/Y')
\xrightarrow{\eqref{eq:BMBC}}
f^*\KBM(X/Y)
\xrightarrow{\eqref{eq:BMsm}}
f^*\one_Y=\one_{Y'}.
\end{align}

\end{definition}

\subsection{}
We now give a compatibility between refined Gysin morphisms and the tensor structure.
\begin{definition}\label{def:Ir}
Let $r\colon S\to S'$ be a smooth morphism and let $X$ be an $S$-scheme. Given a map $\alpha:\KBM(X/S)\to\one_S[n]$, we define a map
\begin{align}
\begin{split}
I_r(\alpha):\KBM(X/S')
=
r_!\KBM(X/S)
\xrightarrow{r_!\alpha}
r_!\one_S[n]
=
\KBM(S/S')[n]
\xrightarrow{\eqref{eq:BMsm}}
\one_{S'}[n].
\end{split}
\end{align}
The map $\alpha\mapsto I_r(\alpha)$ induces an equivalence
\begin{align}
\label{eq:DKSM}
\operatorname{Map}_{\DK(S)}(\KBM(X/S),\one_S[n])
\simeq
\operatorname{Map}_{\DK(S')}(\KBM(X/S'),\one_{S'}[n]).
\end{align}
\end{definition}

Let $X$ and $Y$ be two $S$-schemes. Given two maps $\alpha:\KBM(X/S)\to\one_S[n]$ and $\beta:\KBM(Y/S)\to\one_S[m]$, denote
\begin{align}
\alpha\otimes_S\beta:
\KBM(X\times_SY/S)
\overset{\eqref{eq:BMKun}}{\simeq}
\KBM(X/S)\otimes\KBM(Y/S)
\xrightarrow{\alpha\otimes\beta}
\one_S[n+m].
\end{align}
On the other hand, since $r$ is smooth of relative dimension $d$, the diagonal morphism $\delta:S\to S\times_{S'}S$ is a regular immersion of codimension $d$. By Definition~\ref{def:refined-Gysin-motive}, we associate to the Cartesian square
\begin{align}
\begin{split}
  \xymatrix@=10pt{
    X\times_SY \ar[r]^-{} \ar[d]_-{} & X\times_{S'}Y \ar[d]^-{\pi} \\
    S \ar[r]^-{\delta} & S\times_{S'}S
  }
\end{split}
\end{align}
a map 
\begin{align}
\label{eq:rGdia}
R_{\pi}(\delta)\colon\KBM(X\times_SY/S')
\xrightarrow{\eqref{eq:rGKBM}}
\KBM(X\times_{S'}Y/S').
\end{align}

\begin{lemma}\label{lem:tensor-refined-gysin}
The map $I_r(\alpha\otimes_S\beta):\KBM(X\times_SY/S')\to \one_{S'}[n+m]$ is homotopic to the composition
\begin{align}
\begin{split}
\KBM(X\times_SY/S')
\xrightarrow{\eqref{eq:rGdia}}
\KBM(X\times_{S'}Y/S')
\xrightarrow{I_r(\alpha)\otimes_{S'}I_r(\beta)}
\one_{S'}[n+m].
\end{split}
\end{align}
% and $()\circ R_{\pi}(\delta)$ are homotopic as maps $$.
\end{lemma}

\proof
%Let $f\colon X\to S$ and $g\colon Y\to S$ be structural, and put
%\[
%A:==f_!\one_X,\qquad B:==g_!\one_Y.
%\]
Let $T:=S\times_{S'}S$, with projections $p_1,p_2\colon T\to S$ and structural morphism $\mu\colon T\to S'$. %Thus $r=\mu\circ\delta$, $p_1\circ\delta=p_2\circ\delta=\id_S$. 
By the base change formula~\eqref{eq:BMBC} we have
\begin{align}
\KBM(X\times_{S'}S/T)\simeq p_1^*A,
\qquad
\KBM(S\times_{S'}Y/T)\simeq p_2^*B.
\end{align}
Let $E:=p_1^*\KBM(X/S)\otimes p_2^*\KBM(Y/S)\in\DK(T)$. Then $\delta^*E\simeq \KBM(X/S)\otimes \KBM(Y/S)$, and by the K\"unneth formula~\eqref{eq:BMKun} we have
\begin{align}
\KBM(X\times_SY/S')\simeq\mu_!\delta_!\delta^*E,
\qquad
\KBM(X\times_{S'}Y/S')\simeq\mu_!E.
\end{align}
The refined Gysin morphism~\eqref{eq:rGKBM} is $\mu_!(\gamma_E)$, where
\begin{align}
\gamma_E\colon
\delta_!\delta^*E
\simeq
\delta_!(\delta^!\one_T\otimes\delta^*E)
\simeq
\delta_!\delta^!\one_T\otimes E
\xrightarrow{}E.
\end{align}
%The first isomorphism is absolute purity and the second is the projection formula.
Via the adjoint pair $(r_!,r^!)$, the map $I_r(\alpha\otimes\beta)$ corresponds to
\begin{equation}\label{eq:Ir-first}
\delta^*E[-n-m]
\xrightarrow{\delta^*(p_1^*\alpha\otimes p_2^*\beta)}
\delta^*\one_T
\overset{\eqref{eq:relpurK}}{\simeq}
\delta^!\one_T
\simeq 
r^!\one_{S'}.
\end{equation}
Via the adjoint pair $(\mu_!,\mu^!)$, the map $I_r(\alpha)\otimes I_r(\beta)$ corresponds to
\begin{align}
E[-n-m]%=p_1^*A\otimes p_2^*B
\xrightarrow{p_1^*\alpha\otimes p_2^*\beta}\one_T
\overset{\eqref{eq:relpurK}}{\simeq}
\mu^!\one_{S'}.
\end{align}
Consequently, the map $(I_r(\alpha)\otimes I_r(\beta))\circ R_{\pi}(\delta)$ corresponds to
\begin{align}
\begin{split}
\label{eq:Ir-second}
\delta^*E
\xrightarrow{\eqref{eq:relpurK}}
\delta^!E
\xrightarrow{\delta^!(p_1^*\alpha\otimes p_2^*\beta)}
\delta^!\one_T[n+m]
\simeq 
r^!\one_{S'}[n+m].
\end{split}
\end{align}
The maps \eqref{eq:Ir-first} and \eqref{eq:Ir-second} are homotopic by naturality of the transformation $\delta^*\xrightarrow{\eqref{eq:relpurK}}\delta^!$ (see \cite[Example 4.3.5]{DJK}), as expressed by the commutative diagram
\begin{align}
\begin{split}
  \xymatrix@R=10pt{
    \delta^*E \ar[r]^-{\eqref{eq:relpurK}} \ar[d]_-{\delta^*(p_1^*\alpha\otimes p_2^*\beta)} & \delta^!E \ar[d]^-{\delta^!(p_1^*\alpha\otimes p_2^*\beta)} \\
    \delta^*\one_T \ar[r]^-{\sim}_-{\eqref{eq:relpurK}} & \delta^!\one_T
  }
\end{split}
\end{align}
which finishes the proof.
\endproof

\section{Borel-Moore homology in $\KGL$-modules}
\label{sec:DKBM}
\begin{definition}
For any separate morphism of finite type $f:X\to S$, we denote by 
\begin{align}
\begin{split}
\DK^{BM}(X/S)
&=
\operatorname{Map}_{\DK(X)}(\one_X,f^!\one_S)\\
&=
\operatorname{Map}_{\DK(X)}(\KBM(X/S),\one_S)
\end{split}
\end{align}
the Borel-Moore mapping spectrum in $\DK$ (see \cite[Def. 2.2.1]{DJK}). We also denote
\begin{align}
\DK^{BM}_m(X/S)
=
\pi_m\DK^{BM}(X/S).
\end{align}
\end{definition}

The functoriality of this object is studied in details in \cite{DJK}. In particular, it has proper push-forwards, refined Gysin morphisms and products.
\begin{enumerate}
\item
If $f:X\to Y$ is a proper morphism, the map~\eqref{eq:BMprop} induces a push-forward map
\begin{align}
\label{eq:DKpf}
\DK^{BM}(X/S)
\to
\DK^{BM}(Y/S).
\end{align}
\item
Consider a Cartesian square
\begin{align}
\begin{split}
  \xymatrix@=10pt{
    X' \ar[r]^-{} \ar[d]_-{} & Y' \ar[d]^-{f} \\
    X \ar[r]^-{g} & Y
  }
\end{split}
\end{align}
where $g$ is lci. The map~ \eqref{eq:rGKBM} induces a refined Gysin map
\begin{align}
\label{eq:DKrG}
R_f(g):
\DK^{BM}(Y'/S)
\to
\DK^{BM}(X/S).
\end{align}

\item
If $X$ and $Y$ are two $S$-schemes, there is a map
\begin{align}
\begin{split}
\label{eq:DKprod}
&\DK^{BM}(X/S)
\wedge
\DK^{BM}(Y/S)\\
\xrightarrow{\eqref{eq:BMBC}}
&\DK^{BM}(X\times_SY/Y)
\wedge
\DK^{BM}(Y/S)
\to
\DK^{BM}(X\times_SY/S)
\end{split}
\end{align}
where the last map is defined in \cite[2.2.7 (4)]{DJK}.
\end{enumerate}

\begin{theorem}
\label{thm:Grep}
Let $S$ be a regular scheme. Then for any quasi-projective $S$-scheme $X$, there is an equivalence as $K(S)$-modules
\begin{align}
\label{eq:GBM}
\DK^{BM}(X/S)\simeq G(X)
\end{align}
which is compatible with projective push-forwards, refined Gysin morphisms and products on $G_0$. % and products.
\end{theorem}
%Note that Theorem~\ref{thm:Grep} is proved in \cite{Jin18} in the category $\SH$. Here we work in the category $\KGL$-modules, meaning that we need to work $\KGL$-linearly; we restrict to the quasi-projective case, using an argument simpler than that of \cite{Jin18}.

\subsection{}
If $S$ is a regular scheme, a particular case of the equivalence~\eqref{eq:DKones} becomes
\begin{align}
\label{eq:DKone}
\operatorname{Map}_{\DK(S)}(\one_S,\one_S)
\simeq
G(S).
\end{align}
By \cite[\S13.1]{CD}, the equivalence~\eqref{eq:DKone} is compatible with smooth pullbacks: if $f:T\to S$ is a smooth morphism, there is a commutative diagram
\begin{align}
\label{eq:DKpb}
\begin{split}
  \xymatrix@=10pt{
    \operatorname{Map}_{\DK(S)}(\one_S,\one_S) \ar[r]^-{\eqref{eq:DKone}}_-{\sim} \ar[d]_-{} & G(S) \ar[d]^-{\eqref{eq:Gpb}} \\
    \operatorname{Map}_{\DK(T)}(\one_{T},\one_T) \ar[r]^-{\eqref{eq:DKone}}_-{\sim} & G(T).
  }
\end{split}
\end{align}

\subsection{}
\label{num:BMGcl}
We have the following generalization of~\eqref{eq:DKone}: by \cite[(13.4.1.3)]{CD}, if $S$ is a regular scheme and $i:Z\to S$ is a closed immersion, then there is an equivalence
\begin{align}
\label{eq:CD13413}
\DK^{BM}(Z/S)
\simeq
\operatorname{Map}_{\DK(Z)}(\one_Z,i^!\one_S)
\simeq
K(S\textrm{\ on\ }Z)
\overset{\eqref{eq:KsuppG}}{\simeq}
G(Z).
\end{align}
Indeed, if $Z=S$, this follows from the construction of $\KGL_S$; in the general case, the result follows by identifying both sides as homotopy fibers
\begin{align}
\operatorname{fib}\left(\operatorname{Map}_{\DK(S)}(\one_S,\one_S)\to\operatorname{Map}_{\DK(S-Z)}(\one_{S-Z},\one_{S-Z})\right)
\simeq
\operatorname{fib}\left(K(S)\to K(S-Z)\right).
\end{align}
\begin{lemma}
\label{lm:BMsmr}
Assume that the closed immersion $i:Z\to S$ factors as $Z\xrightarrow{k}M\xrightarrow{p}S$ with $k$ a closed immersion and $p$ smooth. Then the following diagram is commutative:
\begin{align}
\begin{split}
  \xymatrix@=10pt{
    \DK^{BM}(Z/S) \ar[r]^-{\eqref{eq:CD13413}}_-{\sim} \ar[rd]_-{\eqref{eq:DKSM}}^-{\sim} & G(Z)  \\
    & \DK^{BM}(Z/M). \ar[u]_-{\eqref{eq:CD13413}}^-{\wr}
  }
\end{split}
\end{align}

\end{lemma}
\proof
The functor $p^*$ for the smooth morphism $p$ has a left adjoint $p_\#$, and by adjunction and relative purity we have an isomorphism
\begin{align}
\label{eq:relpurr}
p_\#\simeq p_!\operatorname{Th}(T_p).
\end{align}
Therefore we obtain a commutative diagram
\begin{align}
\begin{split}
  \xymatrix@=10pt{
    \operatorname{Map}_{\DK(Z)}(\one_Z,k^!\one_M) \ar[r]^-{\eqref{eq:CD13413}}_-{\sim} \ar[d]^-{\wr} & G(Z) \ar[ddd] \\
    \operatorname{Map}_{\DK(S)}(p_\#k_!\one_Z,\one_S) \ar[d]^-{\wr}_-{\eqref{eq:relpurr}} & \\
    \operatorname{Map}_{\DK(S)}(p_!(k_!\one_Z\otimes\operatorname{Th}(T_p)),\one_S) \ar[d]^-{\wr} & \\
     \operatorname{Map}_{\DK(S)}(p_!k_!\operatorname{Th}((T_p)_{|Z}),\one_S) \ar[d]_-{\eqref{eq:perK}}^-{\wr} \ar[r]^-{\sim} & \operatorname{fib}\left(G(\mathbb{P}((T_p)_{|Z}\oplus\mathcal{O}_Z)\to G(\mathbb{P}(T_p)_{|Z})\right) \ar[d] \\
     \operatorname{Map}_{\DK(Z)}(\one_Z,i^!\one_S) \ar[r]^-{\sim} & G(Z). \\
  }
\end{split}
\end{align}
The left vertical map agrees with~\eqref{eq:DKSM} by functoriality of relative purity, and the right vertical map is the identity map of $G(Z)$ by projective bundle formula for $G$-theory (\cite[\S7 4.3]{Qui}), which finishes the proof.
\endproof

\subsection{Proof of Theorem~\ref{thm:Grep}}
We first construct the equivalence~\eqref{eq:GBM}. First consider the case where $f:X\to S$ is smooth. In this case by representability of algebraic $K$-theory in $\SH$ (see \cite[(13.3.4.2)]{CD}), we obtain
\begin{align}
\begin{split}
\DK^{BM}(X/S)
=
\operatorname{Map}_{\DK(X)}(\one_X,f^!\one_S)
\overset{\eqref{eq:relpurK}}{\simeq}
\operatorname{Map}_{\DK(X)}(\one_X,\one_X)
\overset{\eqref{eq:CD13413}}{\simeq}
G(X).
\end{split}
\end{align}
Now assume that $X$ is quasi-projective over $S$. Then the structure morphism $f:X\to S$ factors as $X\xrightarrow{i}M\xrightarrow{p}S$, where $i$ is a closed immersion and $p$ is smooth. Then there is an isomrophism
\begin{align}
\label{eq:DKG}
\begin{split}
\DK^{BM}(X/S)
\overset{\eqref{eq:DKSM}}{\simeq}
\DK^{BM}(X/M)
\overset{\eqref{eq:CD13413}}{\simeq}
G(X).
\end{split}
\end{align}

\begin{lemma}
The equivalence~\eqref{eq:DKG} is, up to homotopy, independent of the factorization of $f$.
\end{lemma}
\proof
Let $X\xrightarrow{i'}M'\xrightarrow{p'}S$ be another factorization of $f$. Then we have a closed immersion $X\xrightarrow{(i,i')}M\times_SM'$, and by Lemma~\ref{lm:BMsmr}, we have the following commutative diagram:
\begin{align}
\begin{split}
  \xymatrix@=10pt{
  & \DK^{BM}(X/S) & \\
    \DK^{BM}(X/M) \ar[d]_-{\eqref{eq:CD13413}}^-{\wr} \ar[ru]^-{\eqref{eq:DKSM}}_-{\sim} & \DK^{BM}(X/M\times_SM') \ar[r]^-{\sim}_-{\eqref{eq:DKSM}} \ar[l]^-{\eqref{eq:DKSM}}_-{\sim} \ar[d]^-{\eqref{eq:CD13413}}_-{\wr} & \DK^{BM}(X/M') \ar[d]^-{\eqref{eq:CD13413}}_-{\wr} \ar[lu]^-{\sim}_-{\eqref{eq:DKSM}} \\
    G(X) \ar@{=}[r]_-{} & G(X) \ar@{=}[r]_-{} & G(X)
    }
\end{split}
\end{align}
which finishes the proof.
\endproof

\begin{lemma}
\label{lm:DKGrG}
Consider a Cartesian square
\begin{align}
\begin{split}
  \xymatrix@=10pt{
    X' \ar[r]^-{} \ar[d]_-{} & Y' \ar[d]^-{f} \\
    X \ar[r]^-{g} & Y
  }
\end{split}
\end{align}
where $g$ is quasi-projective lci, and let $r\colon Y'\to S$ be a quasi-projective morphism. Then the following diagram commutes:
\begin{align}
\label{eq:DKGrG}
\begin{split}
  \xymatrix@=10pt{
\DK^{BM}(Y'/S) \ar[r]^-{\eqref{eq:DKrG}} \ar[d]_-{\eqref{eq:DKG}}^-{\wr} & \DK^{BM}(X'/S) \ar[d]^-{\eqref{eq:DKG}}_-{\wr} \\
G(Y') \ar[r]^-{\eqref{eq:refGys}} & G(X').
    }
\end{split}
\end{align}
\end{lemma}

\proof
We may assume that $g$ is a regular closed immersion or a quasi-projective smooth morphism. If $g$ is smooth, the result follows from the commutative diagram~\eqref{eq:DKpb}, and indeed a version with support as in~\ref{num:BMGcl}.

Now assume that $g$ is a regular closed immersion. Let $N'=N_XY\times_YY'$. Following the refined constructions in Definitions~\ref{def:refGys} and~\ref{def:refDK} and Lemma~\ref{cor:spnew}, we break the diagram~\eqref{eq:DKGrG} into $3$ parts:
\begin{align}
%\label{eq:DKGrG}
\begin{split}
  \xymatrix@=10pt{
\DK^{BM}(Y'/S) \ar[r]^-{\eqref{eq:BMsm}} \ar[d]_-{\eqref{eq:DKG}}^-{\wr} & \DK^{BM}(Y'\times\mathbb{G}_m/S) \ar[r]^-{\eqref{eq:spKBM}} \ar[d]_-{\eqref{eq:DKG}}^-{\wr} & \DK^{BM}(N'/S) \ar[d]_-{\eqref{eq:DKG}}^-{\wr} & \DK^{BM}(X'/S) \ar[d]^-{\eqref{eq:DKG}}_-{\wr} \ar[l]_-{\eqref{eq:BMsm}}^-{\sim} \\
G(Y') \ar[r]^-{\eqref{eq:Gpb}} & G(Y'\times\mathbb{G}_m) \ar[r]^-{\eqref{eq:spnew}} & G(N') & G(X'). \ar[l]_-{\eqref{eq:Gpb}}^-{\sim}
    }
\end{split}
\end{align}
The left and right squares commute by the smooth case above. For the commutativity of the middle square, one need to show the compatibility of the equivalence~\eqref{eq:DKG} with specializations. Tracking back the construction in~\ref{num:spmap}, this follows from the compatibility of the equivalence~\eqref{eq:DKG} with the localizaiton triangle~\eqref{eq:KBMloc}: this follows from the smooth case above and the case of closed immersions in Lemma~\ref{lm:DKGpf} below.
\endproof

\begin{lemma}
\label{lm:DKGpf}
Let $g:Y\to X$ be a projective morphism. Then the following diagram commutes:
\begin{align}
\begin{split}
  \xymatrix@=10pt{
\DK^{BM}(Y/S) \ar[r]^-{\eqref{eq:BMprop}} \ar[d]_-{\eqref{eq:DKG}}^-{\wr} & \DK^{BM}(X/S) \ar[d]^-{\eqref{eq:DKG}}_-{\wr} \\
G(Y) \ar[r]^-{\eqref{eq:Gpf}} & G(X).
    }
\end{split}
\end{align}
\end{lemma}

\proof
Choose a factorization of $f$ as $X\xrightarrow{i}M\xrightarrow{p}S$, where $i$ is a closed immersion and $p$ is smooth. 
We may assume that $g$ is a closed immersion or the projection of a projective space. If $g$ is a closed immersion, by localization sequence~\eqref{eq:fibG}, it follows from the construction in~\ref{num:BMGcl} that we have a commutative diagram
\begin{align}
\begin{split}
  \xymatrix@=10pt{
\DK^{BM}(Y/M) \ar[r]^-{\eqref{eq:BMprop}} \ar[d]_-{\eqref{eq:CD13413}}^-{\wr} & \DK^{BM}(X/M) \ar[d]^-{\eqref{eq:CD13413}}_-{\wr} \\
G(Y) \ar[r]^-{\eqref{eq:Gpf}} & G(X)
    }
\end{split}
\end{align}
and the result follows from Lemma~\ref{lm:BMsmr}.

Now assume $Y=\mathbb{P}^n_X$. Let $z$ be the point of the $K$-theory spectrum $K(\mathbb{P}^n_X)$ corresponding to the sheaf $\mathcal{O}(-1)$. Consider the following diagram
\begin{align}
\label{eq:totPnDK}
\begin{split}
  \xymatrix@=10pt{
\oplus_{i=0}^n\DK^{BM}(X/S) \ar[r]^-{\sim} \ar[d]_-{\eqref{eq:DKG}}^-{\wr} &\DK^{BM}(\mathbb{P}^n_X/S) \ar[r]^-{\eqref{eq:BMprop}} \ar[d]_-{\eqref{eq:DKG}}^-{\wr} & \DK^{BM}(X/S) \ar[d]^-{\eqref{eq:DKG}}_-{\wr} \\
\oplus_{i=0}^nG(X) \ar[r]^-{\sim}& G(\mathbb{P}^n_X) \ar[r]^-{\eqref{eq:Gpf}} & G(X)
    }
\end{split}
\end{align}
where the two horizontal maps on the left are given by 
\begin{align}
(x_i)_{0\leqslant i\leqslant n}
\mapsto
\sum_{i=0}^nz^i\cdot g^*x_i.
\end{align}
By projective bundle formula for $G$-theory (\cite[\S7 4.3]{Qui}), these maps are isomorphisms. By Lemma~\ref{lm:DKGrG}, the square on the left commutes; by projection formula, $g_*(z^i\cdot g^*x_i)=g_*z^i\cdot x_i$, and the big composed square in~\eqref{eq:totPnDK} commutes, and the result follows.
\endproof

\begin{lemma}
\label{lm:DKG0prod}
Let $S$ be a smooth quasi-projective scheme over a field $k$. Then the equivalence~\eqref{eq:GBM} is compatible with products on $G_0$ in Definition~\ref{def:G0prod}. In other words, if $X$ and $Y$ are two quasi-projective $S$-schemes, we have a commutative diagram
\begin{align}
\label{eq:diagprodG0}
\begin{split}
  \xymatrix@=10pt{
    \DK^{BM}_0(X/S)\times\DK^{BM}_0(Y/S) \ar[r]^-{\eqref{eq:DKprod}} \ar[d]_-{\eqref{eq:DKG}}^-{\wr} & \DK^{BM}_0(X\times_SY/S) \ar[d]^-{\eqref{eq:DKG}}_-{\wr} \\
    G_0(X)\times G_0(Y) \ar[r]^-{\eqref{eq:GproductTor}} & G_0(X\times_SY).
  }
\end{split}
\end{align}
\end{lemma}

\proof
First consider the case where $S=k$. In other words, we need to show that the following diagram commutes:
\begin{align}
\begin{split}
  \xymatrix@=10pt{
    \DK^{BM}_0(X/k)\times\DK^{BM}_0(Y/k) \ar[r]^-{\eqref{eq:DKprod}} \ar[d]_-{\eqref{eq:DKG}}^-{\wr} & \DK^{BM}_0(X\times_kY/k) \ar[d]^-{\eqref{eq:DKG}}_-{\wr} \\
    G_0(X)\times G_0(Y) \ar[r]^-{\eqref{eq:GproductTor}} & G_0(X\times_kY).
  }
\end{split}
\end{align}
If both $X$ and $Y$ are smooth, the result follows from the following commutative diagram:
$$
  \xymatrix@=10pt{
    \DK^{BM}(X/k)\times\DK^{BM}(Y/k) \ar[r]^-{\eqref{eq:DKrG}} \ar[d]_-{\eqref{eq:DKG}}^-{\wr} & \DK^{BM}(X\times_kY/k)\times\DK^{BM}(X\times_kY/k) \ar[d]_-{\eqref{eq:DKG}}^-{\wr} \ar[r]^-{} & \DK^{BM}(X\times_kY/k) \ar[d]^-{\eqref{eq:DKG}}_-{\wr} \\
    G(X)\times G(Y) \ar[r]^-{\eqref{eq:Gpb}} & G(X\times_kY)\times G(X\times_kY) \ar[r]^-{} & G(X\times_kY).
  }
$$
Here note that $G$-theory for smooth schemes agree with $K$-theory by~\eqref{eq:KtoG}, and the horizontal maps of the square on the right arise from the ring structure of $\KGL$ and $K_0$; the square on the left commutes by Lemma~\ref{lm:DKGrG}, and the square on the right commutes by \cite[Example 2.4.1]{Rio10}.

If $X$ and $Y$ are singular, choose closed immersions $X\to M$ and $Y\to N$ with $M$ and $N$ smooth. The result follows from a version with supports of the smooth case. Indeed, we have the following commutative square
$$
  \xymatrix@=10pt{
    \DK^{BM}_0(X/k)\times\DK^{BM}_0(Y/k) \ar[r]^-{\eqref{eq:DKrG}} \ar[d]_-{\eqref{eq:DKG}}^-{\wr} & \DK^{BM}_0(X\times_kN/k)\times\DK^{BM}_0(M\times_kY/k) \ar[d]_-{\eqref{eq:DKG}}^-{\wr} \ar[r]^-{} & \DK^{BM}_0(X\times_kY/k) \ar[d]^-{\eqref{eq:DKG}}_-{\wr} \\
    G_0(X)\times G_0(Y) \ar[r]^-{\eqref{eq:Gpb}} \ar@/_1pc/[rr]_-{\eqref{eq:GproductTor}} & G_0(X\times_kN)\times G_0(M\times_kY) \ar[r]^-{} & G_0(X\times_kY).
  }
$$
The square on the left commutes by Lemma~\ref{lm:DKGrG}, and the square on the right commutes by a version of \cite[Example 2.4.1]{Rio10} with supports. It is straightforward that the lower half-circle commutes, and the result follows.

We now turn to the general case. The result follows from the following commutative diagram:
$$
  \xymatrix@=10pt{
   \DK^{BM}_0(X/S)\times\DK^{BM}_0(Y/S) \ar[rr]^-{\eqref{eq:DKprod}} \ar[d]_-{\eqref{eq:DKSM}}^-{\wr} \ar@/_7pc/[dd]_-{\eqref{eq:DKSM}}^-{\wr} & & \DK^{BM}_0(X\times_SY/S) \ar[d]^-{\eqref{eq:DKSM}}_-{\wr} \ar@/^5pc/[dd]^-{\eqref{eq:DKSM}}_-{\wr} \\
   \DK^{BM}_0(X/k)\times\DK^{BM}_0(Y/k) \ar[r]^-{\eqref{eq:DKprod}} \ar[d]_-{\eqref{eq:DKG}}^-{\wr} &\DK^{BM}_0(X\times_kY/k) \ar[r]^-{\eqref{eq:rGKBM}} \ar[d]_-{\eqref{eq:DKG}}^-{\wr} & \DK^{BM}_0(X\times_SY/k) \ar[d]^-{\eqref{eq:DKG}}_-{\wr}\\
   G_0(X)\times G_0(Y) \ar[r]^-{\eqref{eq:GproductTor}} \ar@/_1pc/[rr]_-{\eqref{eq:GproductTor}} & G_0(X\times_kY) \ar[r]^-{\eqref{eq:refGys}} & G_0(X\times_SY).
  }
$$
Here the horizontal maps in the lower right square are refined Gysin maps associated to the Cartesian square
\begin{align}
\begin{split}
  \xymatrix@=10pt{
    X\times_SY \ar[r]^-{} \ar[d]_-{} & X\times_kY \ar[d]^-{f} \\
    S \ar[r]^-{\delta} & S\times_kS.
  }
\end{split}
\end{align}
\begin{enumerate}
\item
The left and right half-circles follow from Lemma~\ref{lm:BMsmr}.
\item
The upper pentagon follows from Lemma~\ref{lem:tensor-refined-gysin}.
\item
The lower left square is the case $S=k$ above.
\item
The lower right square follow from Lemma~\ref{lm:DKGrG}.
\item
For the lower half-circle, by Lemma~\ref{lm:refGysexp} we have
\begin{align}
\Bigl[\mathcal F\boxtimes_S\mathcal G\Bigr]
=
R_f(\delta)
\Bigl[\mathcal F\boxtimes_k\mathcal G\Bigr]
=
\sum_{q\ge0}
(-1)^q
\Bigl[
\operatorname{Tor}^{\mathcal O_{S\times_k S}}_q
\bigl(
\mathcal F\boxtimes_k\mathcal G,
\mathcal O_S
\bigr)
\Bigr]
\end{align}
and the result then follows from the following formula by~\eqref{eq:Torr}:
\begin{align}
\operatorname{Tor}^{\mathcal O_{S\times_k S}}_q
\bigl(
\mathcal F\boxtimes_k\mathcal G,
\mathcal O_S
\bigr)
\simeq
\operatorname{Tor}^{\mathcal O_S}_q
(\mathcal F,\mathcal G).
\end{align}
\end{enumerate}
\endproof

More generally, Lemma~\ref{lm:DKG0prod} holds for any regular scheme $S$:
\begin{lemma}
\label{lm:DKG0prodd}
Let $S$ be a regular scheme, and let $X$ and $Y$ be two quasi-projective $S$-schemes. Then the diagram~\ref{eq:diagprodG0} commutes.
\end{lemma}
\proof
As in the proof of Lemma~\ref{lm:DKG0prod}, choose closed immersions $X\to M$ and $Y\to N$ with $M$ and $N$ smooth. The result follows from the following commutative square
$$
  \xymatrix@=10pt{
    \DK^{BM}_0(X/S)\times\DK^{BM}_0(Y/S) \ar[r]^-{\eqref{eq:DKrG}} \ar[d]_-{\eqref{eq:DKG}}^-{\wr} & \DK^{BM}_0(X\times_SN/S)\times\DK^{BM}_0(M\times_SY/S) \ar[d]_-{\eqref{eq:DKG}}^-{\wr} \ar[r]^-{} & \DK^{BM}_0(X\times_SY/S) \ar[d]^-{\eqref{eq:DKG}}_-{\wr} \\
    G_0(X)\times G_0(Y) \ar[r]^-{\eqref{eq:Gpb}} \ar@/_1pc/[rr]_-{\eqref{eq:GproductTor}} & G_0(X\times_SN)\times G_0(M\times_SY) \ar[r]^-{} & G_0(X\times_SY).
  }
$$
To see that the lower half-circle commutes, we use the fact that for $\mathcal{F}\in \operatorname{Coh}(X)$ and $\mathcal{G}\in \operatorname{Coh}(Y)$, we have
$\mathcal{F}_{|X\times_SN}\boxtimes^{\mathbf L}_{M\times_SN}\mathcal{G}_{|M\times_SY}\simeq \mathcal{F}\boxtimes^{\mathbf L}_{S}\mathcal{G}\in \operatorname{Coh}(X\times_SY)$.
\endproof

\section{The Chow weight structure in $\KGL$-modules}
\label{sec:Chowwt}
\subsection{}
(See \cite[Def. 3.1]{Sos}) Recall that a \textbf{weight structure} on a stable $\infty$-category $\mathcal C$ consists of two full subcategories $\mathcal C_{w\leq0}$ and $\mathcal C_{w\geq0}$ satisfying the following axioms:
\begin{enumerate}
\item
The subcategories $\mathcal C_{w\leqslant0}$ and $\mathcal C_{w\geqslant0}$ are stable under retracts, and
\begin{align}
\Sigma\mathcal C_{w\leqslant0}\subset \mathcal C_{w\leqslant0},
\qquad
\Omega\mathcal C_{w\geqslant0}\subset \mathcal C_{w\geqslant0}.
\end{align}
\item
For any objects $X\in\mathcal C_{w\leqslant0}$ and $Y\in\mathcal C_{w\geqslant0}$ we have
\begin{align}
\pi_0\operatorname{Map}_{\mathcal C}(X,\Sigma Y)=0.
\end{align}
\item
For any object $X\in\mathcal{C}$, there exist two objects $X_{\leqslant0}\in\mathcal C_{w\leqslant0}$ and $X_{\geqslant1}\in\Sigma\mathcal C_{w\geqslant0}$ and a fiber sequence in $\mathcal{C}$
\begin{align}
X_{\leqslant0}
\to
X
\to
X_{\geqslant1}.
\end{align}
\end{enumerate}
Its \textbf{heart} is the full subcategory $\mathcal C_{w=0}$ spanned by objects that lie in both $\mathcal C_{w\leq0}$ and $\mathcal C_{w\geq0}$. It is \textbf{bounded} if for any object $X\in\mathcal{C}$, there exists an integer $n$ such that $X$ lies in both $\Sigma^n\mathcal C_{w\leq0}$ and $\Omega^n\mathcal C_{w\geq0}$.

\begin{theorem}[\textrm{\cite[Th. 4.3.2]{Bon10}, \cite[Th. 1.9]{Heb}}]
\label{th:wtexist}
Let $\mathcal C$ be an idempotent complete stable $\infty$-category.
Let $\mathcal{H}$ be a collection of objects of $\mathcal C$, such that $\mathcal C$ is generated by $\mathcal{H}$ as a thick stable category (i.e. smallest full subcategory stable under retracts, extensions, suspensions and loops). %$\langle\mathcal{H}\rangle$ be the smallest collection of objects of $\mathcal C$ stable  that contains objects in $\mathcal{H}$. %Assume that every object of $$
Then the following are equivalent:
\begin{enumerate}
\item
There exists a weight structure on $\mathcal C$ such that every object of $\mathcal{H}$ lies in the heart $\mathcal C_{w=0}$;

\item
The collection $\mathcal{H}$ is \textbf{negative}, that is, for any two objects $A,B\in\mathcal C$ and any $n>0$ we have 
\begin{align}
\pi_{-n}\operatorname{Map}(A,B)=0.
\end{align}
\end{enumerate}
If this is the case, such a weight structure is unique and bounded, and the heart $\mathcal C_{w=0}$ agrees with the full subcategory generated by $\mathcal{H}$ under retracts and finite direct sums.
\end{theorem}

\subsection{}
In what follows we consider the case $\mathcal C=\DK_c$ the constructible subcategory. Let $\mathcal{H}$ be the collection of objects of the form $\KBM(X/S)$ in $\DK_c(S)$, where $f:X\to S$ is a projective morphism such that $X$ is regular. 
\begin{lemma}
The collection $\mathcal{H}$ is \textbf{negative}.
\end{lemma}
\proof
If $g:Y\to S$ is another projective morphism with $Y$ regular, we an equivalence
\begin{align}
\label{eq:KBMHom}
\begin{split}
\epsilon_{X,Y}:
&\operatorname{Map}(\KBM(X/S),\KBM(Y/S))
=\operatorname{Map}(\KBM(X/S),g_*\one_Y)\\
\simeq&\operatorname{Map}(g^*\KBM(X/S),\one_Y)
\simeq\operatorname{Map}(\KBM(X\times_SY/Y),\one_Y) \\
=&\DK^{BM}(X\times_SY/Y)\overset{\eqref{eq:DKG}}{\simeq} G(X\times_SY).
\end{split}
\end{align}
Since the $G$-theory spectrum is connective, the result follows.
\endproof

\subsection{}
\label{num:Pnice}
In order to apply Theorem~\ref{th:wtexist}, we need to know that $\mathcal{H}$ generates $\DK_c(S)$. Let $\mathcal{P}$ be a collection of prime numbers. We say that a scheme $S$ is \textbf{$\mathcal{P}$-nice} if there exists a scheme $B$ which is universally $\mathcal{P}$-resolvable in the sense of \cite[4.1.4]{Tem} with $\operatorname{char}(B)\subset\mathcal{P}$, and a maximally dominating morphism of finite type $S\to B$. For example, by \cite[Th. 1.1]{CP}, %the conditions in Lemma~\ref{lm:gen} 
this is the case if $B$ is a separated, reduced scheme of dimension at most $3$, and $\mathcal{P}=\operatorname{char}(B)$. The following result on generators is a generalization of \cite[Cor. 4.4.3]{CD}:
\begin{lemma}
\label{lm:gen}
Let $S$ be a $\mathcal{P}$-nice scheme and let $Z$ be a closed subset of $S$. Let $\mathcal{G}$ be the family of objects of the form $f_*\one_X(n)$ in $\SH(S)[\mathcal{P}^{-1}]$, where $f:X\to S$ is a projective morphism such that $X$ is regular, $n$ is an integer, and $f^{-1}(Z)$ is either empty, or the whole $X$, or the support of a strict normal-crossing divisor. Let $\mathcal{C}$ be the thick stable subcategory of $\SH(S)[\mathcal{P}^{-1}]$ generated by $\mathcal{G}$. Then $\mathcal{C}$ agrees with the subcategory of constructible objects $\SH_c(S)[\mathcal{P}^{-1}]$.
\end{lemma}
The statement also holds with $\SH$ replaced by $\DK$, with the same argument.
\proof
We proceed as in \cite[Lemma 2.3.6]{Jin24}. By \cite[Lemme 2.2.23]{Ayo}, $\SH_c(S)$ is the thick stable subcategory of $\SH(S)$ generated by the elements of the form $f_*\one_X(n)$, where $f:X\to S$ is a projective morphism and $n$ is an integer. We claim that $f_*\one_X(n)$ lies in $\mathcal{C}$ for all such morphism $f$ and integer $n$, and we use induction on the dimension of $X$. By cdh descent (\cite{Cis13}), we may assume $X$ integral. If $X$ is empty, there is nothing to prove. So by induction we may assume that the claim holds for all projective $S$-schemes whose dimension is smaller than $X$. By \cite[Th. 4.3.1]{Tem}, there exists a projective, surjective, generically finite morphism $h:X'\to X$ whose degree has all its prime divisors lie in $\mathcal{P}$, such that $X'$ is integral and regular and $(fh)^{-1}(Z)$ is either empty, or the whole $X'$, or the support of a strict normal crossing divisor. The morphism induces a finite field extension $\widetilde{h}:\operatorname{Spec}(L)\to\operatorname{Spec}(K)$ on generic points, which is a finite separable extension followed by a finite purely inseparable extension. Since we have inverted the degree of $\widetilde{h}$ (in fact the inseparable degree suffices), by \cite[Lemmas B.3 and B.4]{LYZR} and \cite[Th. 2.1.1]{EK}, the canonical map $\one_K\to\widetilde{h}_!\one_L$ has a retraction in $\SH(K)[\mathcal{P}^{-1}]$. By continuity of $\SH$ as explained in the proof of \cite[Lemma 2.4.6]{BD}, there exists a non-empty open subscheme $U$ of $X$ such that for $V=h^{-1}(U)$, $M^{BM}(U/S)$ is a direct summand of $M^{BM}(V/S)$ in $\SH(S)[\mathcal{P}^{-1}]$. On the other hand, by the localization sequence, for any non-empty open subscheme $V$ of $X'$, the object $M^{BM}(V/S)$ lies in $\mathcal{C}$, since $X'-V$ has dimension smaller than $X$. It follows that $M^{BM}(U/S)$ also lies in $\mathcal{C}$. We conclude by using the localization sequence and the induction assumption again.
\endproof

\begin{corollary}[See \textrm{\cite[Th. 2.1.2]{BoL16}}]
\label{cor:BL212}
Let $\mathcal{P}$ be a collection of prime numbers and let $S$ be a $\mathcal{P}$-nice scheme in the sense of~\ref{num:Pnice}.
There is a unique bounded weight structure on $\DK_c(S)[\mathcal{P}^{-1}]$, called the \textbf{Chow weight structure}, whose heart $\DK_{c}^{w=0}(S)[\mathcal{P}^{-1}]$ is the full subcategory generated under retracts and finite direct sums by objects of the form $\KBM(X/S)$, where $f:X\to S$ is a projective morphism such that $X$ is regular. 
\end{corollary}

\subsection{}
%Let $k$ be a perfect field of exponential characteristic $p$ and let $S$ be a separated $k$-scheme of finite type. 
We consider the homotopy category $\operatorname{Ho}(\DK_{c}^{w=0}(S)[\mathcal{P}^{-1}])$. The following is \cite[Prop.~2.39]{Jin16}:
\begin{proposition}\label{prop:composition}
Let $X,Y,Z$ be projective $S$-schemes that are regular, %smooth over $k$, 
and consider two maps
\begin{align}
\alpha\colon\KBM(X/S)\to\KBM(Y/S),
\qquad
\beta\colon\KBM(Y/S)\to\KBM(Z/S).
\end{align}
%Consider the Cartesian diagram
%\begin{align}
%\begin{split}
%  \xymatrix@=10pt{
%    X\times_SY\times_SZ \ar[r]^-{} \ar[d]_-{} & X\times_SY\times_kY\times_SZ \ar[d]^-{f} \\
%    Y \ar[r]^-{\delta} & Y\times_kY.
%  }
%\end{split}
%\end{align}
If $p\colon X\times_SY\times_SZ\to X\times_SZ$ is the canonical projection, then the map $\epsilon_{X,Z}$ in~\eqref{eq:KBMHom} satisfies
\begin{align}
\epsilon_{X,Z}(\beta\circ\alpha)
=
p_*\bigl(\epsilon_{X,Y}(\alpha)\times_S\epsilon_{Y,Z}(\beta)\bigr)
%=
%p_*\circ R_f(\delta)
%\bigl(\epsilon_{X,Y}(\alpha)\times_k\epsilon_{Y,Z}(\beta)\bigr)
.
\end{align}
\end{proposition}

\subsection{}
We are now ready to prove the main theorem, identifying the Chow weight-heart on $\DK_c$:
\begin{theorem}\label{thm:main}
%Let $k$ be a perfect field of exponential characteristic $p$ and let $S$ be a separated $k$-scheme of finite type. 
Let $\mathcal{P}$ be a collection of prime numbers and let $S$ be a $\mathcal{P}$-nice scheme in the sense of~\ref{num:Pnice}. There is an equivalence 
\begin{align}
\KM(S)[\mathcal{P}^{-1}]
\simeq
\operatorname{Ho}(\DK_{c}^{w=0}(S)[\mathcal{P}^{-1}])
\end{align}
between the category of $K_0$-motives over $S$ (Definition~\ref{def:K0M}) and the Chow weight-heart on $\DK_c$ (Corollary~\ref{cor:BL212}).
\end{theorem}
\proof
First construct a functor $F:\KC(S)[\mathcal{P}^{-1}]\to \operatorname{Ho}(\DK_{c}^{w=0}(S)[\mathcal{P}^{-1}])$ by sending an object $f:X\to S$ with $f$ projective and $X$ smooth over $k$ to the object $\KBM(X/S)$. The equivalence~\eqref{eq:KBMHom}, Proposition~\ref{prop:composition} and Lemmas~\ref{lm:DKGrG} %, \ref{lm:DKGpf} 
and \ref{lm:DKG0prodd} show that $F$ is indeed a functor, which is fully faithful. Since the category $\DK(S)$ is idempotent complete, the functor $F$ extends to a functor $\KM(S)[\mathcal{P}^{-1}]\simeq\operatorname{Ho}(\DK_{c}^{w=0}(S)[\mathcal{P}^{-1}])$ which is an equivalence.
\endproof

\end{document}